\documentclass[a4paper,12pt]{amsart}
\usepackage{fullpage}

\usepackage{bm}
\usepackage{etex}
\usepackage[pdftex]{graphicx}
\usepackage[hyphens]{url}
\usepackage{array}
\usepackage{natbib}
\usepackage{amsthm}

\theoremstyle{plain}
\newtheorem{thm}{Theorem}
\newtheorem{prop}[thm]{Proposition}
\theoremstyle{definition}
\newtheorem{defn}{Definition}

\DeclareMathOperator{\Expectation}{\mathbb E}

\DeclareMathOperator{\Hessian}{Hess}
\DeclareMathOperator{\POL}{POL}

\DeclareMathOperator{\Var}{Var}
\DeclareMathOperator{\diag}{diag}
\DeclareMathOperator{\natnabla}{\widetilde\nabla}
\newcommand{\absoluteval}[1]{\left\vert#1\right\vert}
\newcommand{\derivby}[1]{\frac{d}{d#1}}
\newcommand{\diagof}[1]{\diag\left(#1\right)}
\newcommand{\euler}{\mathrm{e}}
\newcommand{\expectat}[2]{{\mathbb E}_{#1}\left[#2\right]}
\newcommand{\expectof}[1]{\Expectation\left[#1\right]}
\newcommand{\expof}[1]{\exp\left(#1\right)}

\newcommand{\hessianof}[1]{\Hessian #1}
\newcommand{\logof}[1]{\log\left(#1\right)}
\newcommand{\polof}[1]{\POL\left(#1\right)}

\newcommand{\reals}{\mathbb R}
\newcommand{\scalarat}[3]{\left\langle#2,#3\right\rangle_{#1}}
\newcommand{\setof}[2]{\left\{#1\middle|#2\right\}}
\newcommand{\set}[1]{\left\{#1\right\}}

\newcommand{\varat}[2]{\Var_{#1}\left(#2\right)}

\title[Polarization]{The gradient flow of the polarization measure. \\ With an appendix.}

\author[G. Pistone]{Giovanni Pistone}
\address{G. Pistone: de Castro Statistics, Collegio Carlo Alberto, Via Real Collegio 30, 10024 Moncalieri, Italy}
\email{giovanni.pistone@carloalberto.org}
\urladdr{www.giannidiorestino.it}

\author[M.-P. Rogantin]{Maria Piera Rogantin}
\address{M.-P. Rogantin: Dipartimento di Matematica, Via Dodecaneso, 35, 16146 Genova, Italy}
\email{rogantin@dima.unige.it}
\urladdr{www.dima.unige.it/rogantin}

\thanks{ 
\paragraph{Acknowledgments} G. Pistone is supported by De Castro Statistic, Collegio Carlo Alberto, Moncalieri. The Authors whish to thank A. Bacciotti (Politecnico di Torino), M. Gasparini (Politecnico di Torino) and L. Malag\`o (Shinshu University) for helpful suggestions.
} 

\date{\today}

\begin{document}

\begin{abstract}
The polarization measure is the probability that among 3 individuals chosen at random from a finite population exactly 2 come from the same class. This index is maximum at the midpoints of the edges of the probability simplex. We compute the gradient flow of this index that is the differential equation whose solutions are the curves of steepest ascent. Tools from Information Geometry are extensively used. In a time series, a comparison of the estimated velocity of variation with the direction of the gradient field should be a better index than the simple variation of the index.
\end{abstract}

\maketitle
\tableofcontents
\bibliographystyle{plainnat}

\section{Introduction}

Given a discrete distribution $\pi$ on $n+1$ classes $x = 0,1,\dots,n$, we consider an index called \emph{polarization measure}, defined by
\begin{equation}\label{eq:POL}
  \polof{\pi} = \sum_{x=0}^n \pi_x^2(1-\pi_x).
\end{equation}

The polarization measure has been introduced in Economics for real distributions by \citet{esterban|ray:1994}. The discrete version we consider here has been used in \citet[ p. 10]{pino|vidal-robert:2013}.

The polarization measure has the following interpretation. Let $X,Y,Z$ be i.i.d. $\sim\pi$ and consider the indicator of \emph{exactly two equal}
\begin{equation*}
  I_{2} = (X = Y \ne Z) + (X = Z \ne Y) + (Y = Z \ne X).
\end{equation*}
 Then $\expectof{I_2} =  3 \sum_{x=0}^n \pi_x^2(1-\pi_x) = 3 \polof{\pi}$.

\begin{figure} 
\begin{center}
\includegraphics[width=.30\linewidth, viewport = 66 87 467 435]{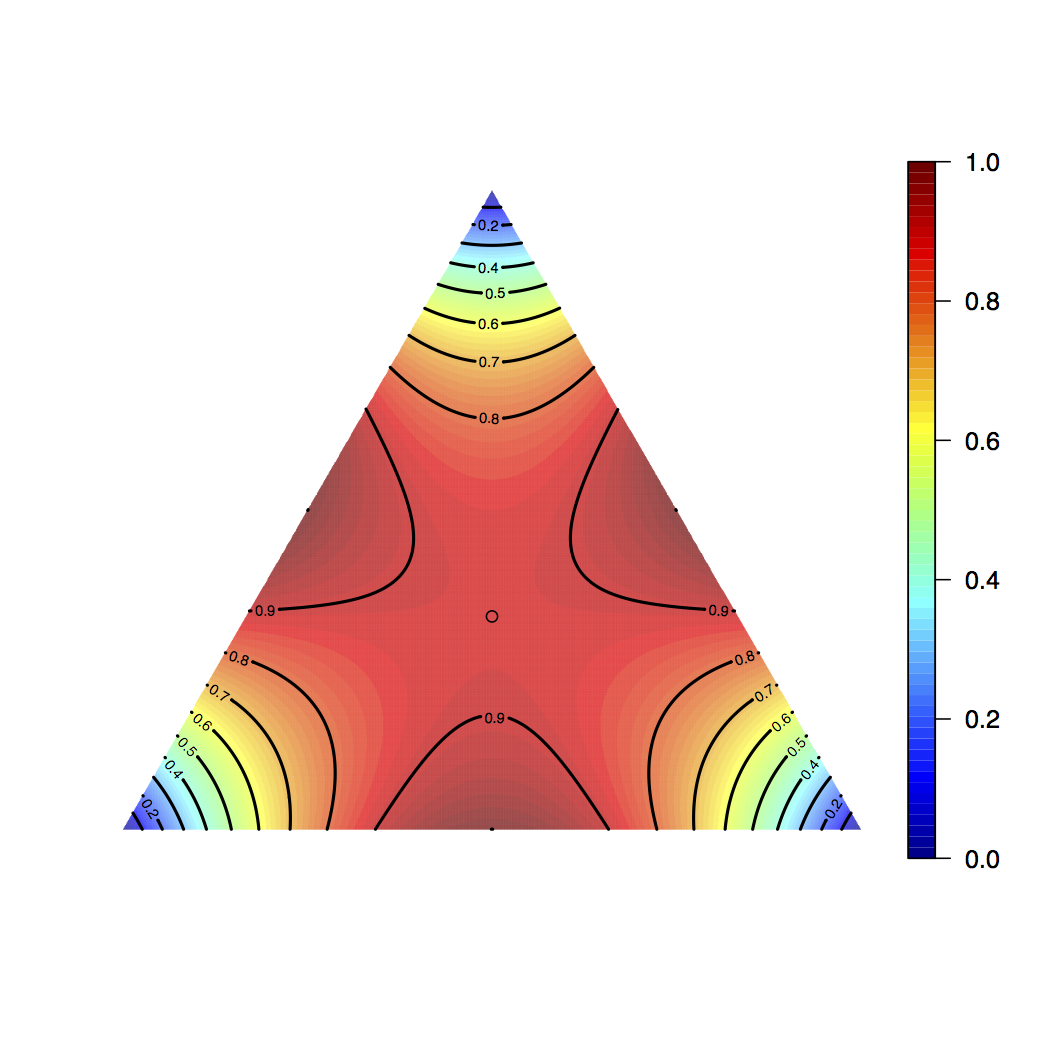}
\caption{Normalised Polarization. The display shows: the probability simplex as an equilateral triangle; the level curves of $4 \times \POL$; the unstable critical point at $\pi = (1/3.1/3,1/3)$ (circle); the minimum points at the vertexes (triangles); the maximum points at $\pi=(0,1/2,1/2), (1/2,0,1/2), (1/2,1/2,0)$ (squares).}
\end{center}
\label{fig:polarization}
\end{figure}

The polarization measure on the classes $\set{0,1,2}$, as shown in  Fig. \ref{fig:polarization}, has an unstable critical point at the uniform distribution, it is zero in the case of concentration in one class, and reaches its maximum 1/4 on distributions on two classes with equal probabilities. Polarization measure was devised to be an index of the distance of a distribution from the three cases of maximal polarization. In Fig. \ref{fig:polarization}  the simplex is represented as an equilateral triangle. In the following we shall use different sets of coordinates to represent the probability simplex, e.g. see Fig. \ref{fig:gradpol-A} (left).

We want to study the dynamics of this index, i.e. to characterise evolutions that maximise or minimise the  index. This study requires tools from Information Geometry (IG) e.g., \citet{amari|nagaoka:2000,gibilisco|pistone:98, pistone|rogantin:99,gibilisco|riccomagno|rogantin|wynn:2010,pistone:2013GSI,malago|pistone:2014Entropy}. However, the following presentation is actually largely self-contained.

The recourse to IG is not dispensable because the ordinary gradient flow of $\POL$, as shown in Fig. \ref{fig:gradpol-B} (left),  does not lead to the extrema of interest on the border of the probability simplex. Consequently, one wants to turn to a different way to compute the gradient, i.e. to the so-called Amari's natural gradient. We use elementary fact of the theory of Dynamical Systems to characterise critical points of the gradient flow and refer to \citet{abraham|marsden|ratiu:1988}.

The basics of IG are discussed in Sec. \ref{sec:natural gradient}. The application of IG to the polarization measure is described in Sec. \ref{sec:pol}.
 The possibility of a generalisation of such an index is shortly discussed in Sec. \ref{sec:POLs}, while the reduction of the problem to the study of an exponential family is presented in Sec. \ref{sec:POLasE}. We suggest a possible application in Sec. \ref{sec:app}. 

Further material, not directly related with the measure of polarization, but suggested by the methodology, is presented in the Appendixes. Differential equations on the probability simplex are well known in applications other then Descriptive Statistics. We briefly discuss the relations between these applications and our one in App. \ref{sec:replicator}. In App. \ref{sec:hessian} some issues related to the second order calculus are briefly discussed. 

\section{Natural gradient}
\label{sec:natural gradient}
We denote by $\Delta_n$ the simplex of the probability function $\pi$ on $0,1,\dots,n$. The interior of the simplex, $\Delta_n^\circ $,  is the set of the strictly positive probability functions,
\begin{equation*}
 \Delta_n^\circ = \setof{\pi \in \reals^{n+1}}{\sum_{x=0}^n \pi_x = 1, \pi_x > 0, x=0,1,\dots,n}.
\end{equation*}

The border of the simplex is the union of all the faces of $\Delta_n$ as a convex set. We recall that  a face of maximal dimension $n-1$ is called facet. A facet is a simplex of dimension $n-1$.

We define $B_\pi$ to  be the vector space of random variables $U$ that are $\pi$-centered, $\expectat \pi U = 0$. In the geometry of $\reals^{n+1}$, $B_\pi$ is the plane through the origin, orthogonal to the vector $\overrightarrow{O\pi}$.

\begin{defn}\
\begin{enumerate}
\item
The \emph{tangent bundle} of the open simplex $\Delta^\circ$ is the set
\begin{equation*}
  T\Delta_n^\circ = \setof{(\pi,U)}{\pi \in \Delta_n^\circ, U \in  B_\pi}.
\end{equation*}
\item
If $I \ni t \mapsto p(t) \in \Delta_n^\circ$ is a one-dimensional statistical model, geometrically a curve, its \emph{score}
\begin{equation*}
  D p(t) = \frac{\dot p(t)}{p(t)} = \derivby t \log p(t)
\end{equation*}
belongs to $B_{p(t)}$ for all $t \in I$. As the score is a centered random variable, hence $I \ni t \mapsto (p(t),D p(t))$ is a curve in the tangent bundle.
\end{enumerate}
\end{defn}

In fact, $U \in B_{\pi}$ is meant to represent a generic velocity vector through $\pi$, see Fig. \ref{fig:tangentbundle}. The score is a representation of the velocity along a curve, because of  a geometric interpretation of C. R. Rao's classical computation:
\begin{multline}\label{eq:rao}
\frac d {dt} E_t[U]= \frac d {dt} \sum_x U(x)p(x;t)= \sum_x U(x)  \frac d {dt} p(x;t)=\\ \sum_x U(x)  \frac d {dt} \log\left( p(x;t)\right)p(x;t)=
\sum_x \left(U(x)-E_t[U]\right)  \frac d {dt} \log\left( p(x;t)\right)p(x;t)=\\
E_t\left[\left(U-E_t[U]\right) \frac d {dt} \log\left( p(t)\right) \right]=
 \scalarat {p(t)} {U-E_t[U]}{D p(t)}
\end{multline}
We observe that the scalar product above is the scalar product on $B_{p(t)}$.

A curve on the simplex is a parametric model. The  probability $\pi$  is represented by a vector from $O$ to the point whose coordinates are $\left(\pi_i\right)_{i=0\dots,n }$. In Fig. \ref{fig:tangentbundle}, the   velocity vectors are represented by arrows; they are orthogonal to the vectors $\overrightarrow{O\pi}$.

\begin{figure}
\label{fig:tangentbundle}
\centering \includegraphics[width=0.30\linewidth, viewport = 67 61 238 233]{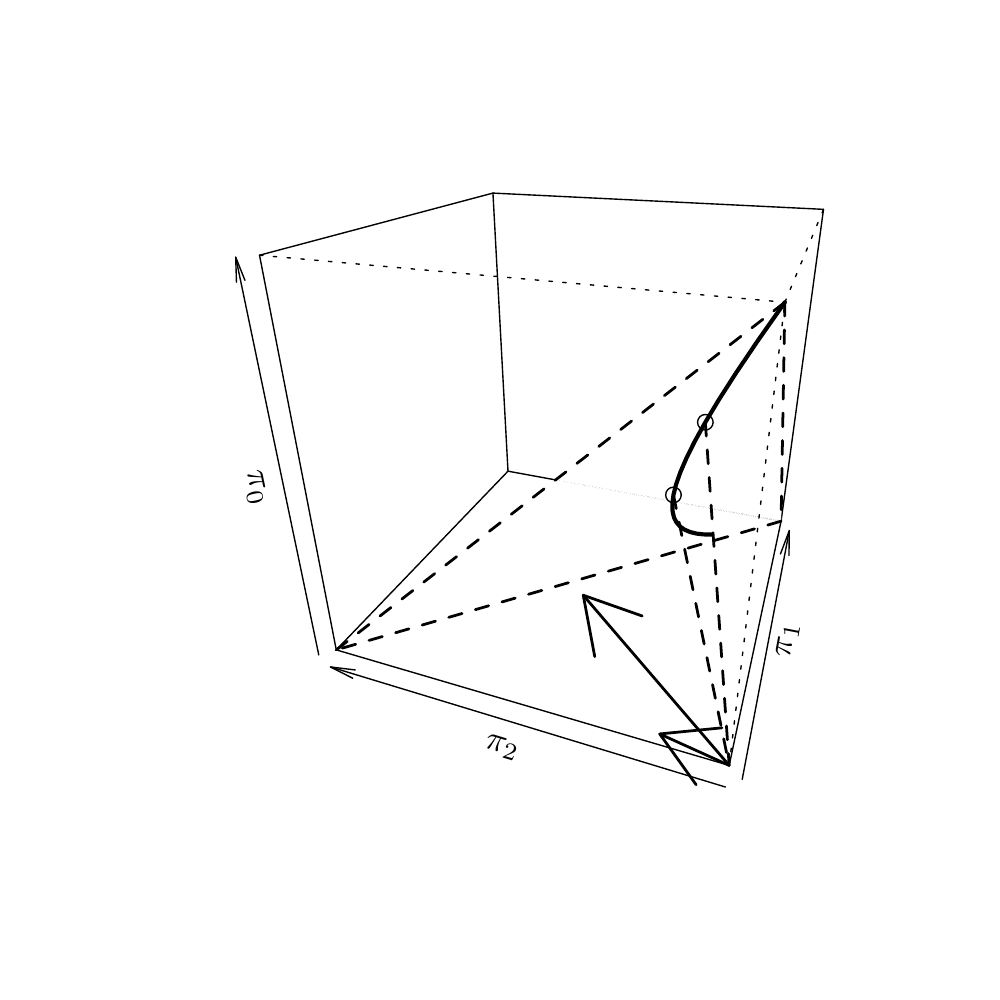}
\caption{The simplex (solid triangle) is view from below. The curve on the simplex is a parametric model. The  probabilities $\pi$  are represented by vectors from $O$ to the point whose coordinates are $\pi$. The   velocity vectors are represented by arrows; they are orthogonal to the vectors from $O$ to $\pi$.}\end{figure}

In our context, a \emph{vector field} $F \colon \Delta_n^\circ \to \reals^{n+1}$ is a mapping such that $F(\pi) \in B_\pi$, i.e. such that the couple $(\pi, F(\pi))$ belongs to the tangent bundle. $(\pi,F(\pi)) \in T\Delta_n^\circ$ for all $\pi$.
 Because of our geometrical construction, here we prefer to call $F$ a vector field, but we want to stress its statistical meaning of centered function of the distribution.

A \emph{differential equation} is an equation of the form $D p(t) = F(p(t))$.

Given a real function $\phi \colon \Delta_n^\circ \to \reals$, its \emph{gradient} is the vector field $\nabla \phi$ such that for all curves $p(\cdot)$ we have
\begin{equation}
\label{eq:gradient}
\frac{d}{dt} \phi(p(t)) = \scalarat {p(t)} {\nabla\phi(p(t))} {D p(t)}, \quad \scalarat {\pi} u v = \expectat {\pi} {uv}, u,v \in B_\pi.
\end{equation}

The Rao's computation in Eq. \eqref{eq:rao} is the prototypical gradient computation.

The \emph{gradient flow equation} is the differential equation
\begin{equation*}
  D p(t) =  \nabla \phi(p(t)).
\end{equation*}

Along a solution of the gradient flow equation the value of $\phi$ is increasing because $d \phi(p(t)) / dt = \scalarat {p(t)} {\nabla\phi(p(t))} {\nabla \phi(p(t)} \ge 0$. Actually the solution of the gradient flow equation is the curve of \emph{steepest ascent}.

Computations are usually performed in a \emph{parametrization}
\begin{equation*}
  \pi \colon \Theta \ni \bm \theta \mapsto \pi(\bm\theta) \in \Delta_n^\circ,
\end{equation*}
$\Theta$ being an open set in $\reals^n$. The $j$-th coordinate curve is obtained by fixing the other $n-1$ components and moving $\theta_j$ only. The scores of the $j$-th coordinate curves are the random variables
\begin{equation*}
  D_j \pi(\bm \theta) = \frac{\partial}{\partial \theta_j} \log \pi(\bm \theta), \quad j = 1,\dots,n.
\end{equation*}

The sequence $(D_j \pi(\bm \theta) \colon j = 1,\dots,n)$ is a vector basis of the tangent space $B_{\pi(\bm\theta)}$. The representation of the scalar product in such a basis is
\begin{equation*}
  \scalarat{\pi(\bm\theta)}{\sum_{i=1}^n \alpha_i D_i \pi(\bm \theta)}{\sum_{j=1}^n \beta_j D_j \pi(\bm \theta)} = \sum_{i,j=1}^n \alpha_i\beta_j I_{ij}(\bm\theta),
\end{equation*}
where the matrix $I(\bm\theta) = \left[\scalarat{\pi(\bm\theta)}{D_i \pi(\bm \theta)}{D_j \pi(\bm \theta)}\right]_{i,i=1}^n$ is the \emph{Fisher information} matrix.

If $\bm\theta \mapsto \widetilde \phi(\bm\theta)$ is the expression in the parameters of a function $\phi \colon \Delta_n^\circ \to \mathbb R$, that is $\widetilde \phi(\bm\theta) = \phi(\pi(\bm\theta))$,  and $t \mapsto \bm\theta(t)$ is the expression in the parameters of a generic curve $p \colon I \to \Delta_n^\circ$, then the components of the gradient in \eqref{eq:gradient} are expressed in terms of the ordinary gradient by observing that
\begin{equation*}
  \frac{d}{dt} \phi(p(t)) = \frac{d}{dt} \widetilde\phi(\bm\theta(t)) = \sum_{j=1}^n \frac{\partial}{\partial\theta_j} \widetilde\phi(\bm\theta(t)) \dot\theta_j(t).
\end{equation*}

As $D p(t)=\sum_j D_j \pi(\bm\theta(t)) \dot \theta_j(t)$, we obtain from \eqref{eq:gradient}
\begin{equation}\label{eq:1}
  \frac{d}{dt} \phi(p(t)) = \scalarat {p(t)} {\nabla\phi(p(t))} {D p(t)} = \sum_j \scalarat {p(t)} {\nabla\phi(p(t))} {D_j \pi(\bm\theta(t)) \dot \theta_j(t)}.
\end{equation}

\begin{defn}[\cite{amari:1998natural}]
The \emph{natural gradient} is a vector $\widetilde \nabla \widetilde\phi(\bm\theta)$ whose components are the coordinates of the gradient $\nabla \widetilde\phi(\pi(\bm \theta)) \in B_{\pi(\bm\theta)}$ in its $\pi$-basis, that is
\begin{equation*}
 \nabla \phi(\pi(\bm \theta)) = \sum_{j=1}^n (\widetilde \nabla \widetilde\phi(\bm\theta))_j D_j \pi(\bm\theta).
\end{equation*}

By substitution of the expression in \eqref{eq:1} we obtain
\begin{equation}\label{def:nat-grad}
  \widetilde \nabla \widetilde\phi(\bm\theta) = \nabla \widetilde\phi(\bm\theta) I^{-1}(\bm\theta).
\end{equation}
\end{defn}

Fig. \ref{fig:albero} is an illustration for the function $\phi(\pi) = \POL(\pi)$.

\begin{figure}
\centering \includegraphics[width=0.70\textwidth, viewport = 20 125 822 476]{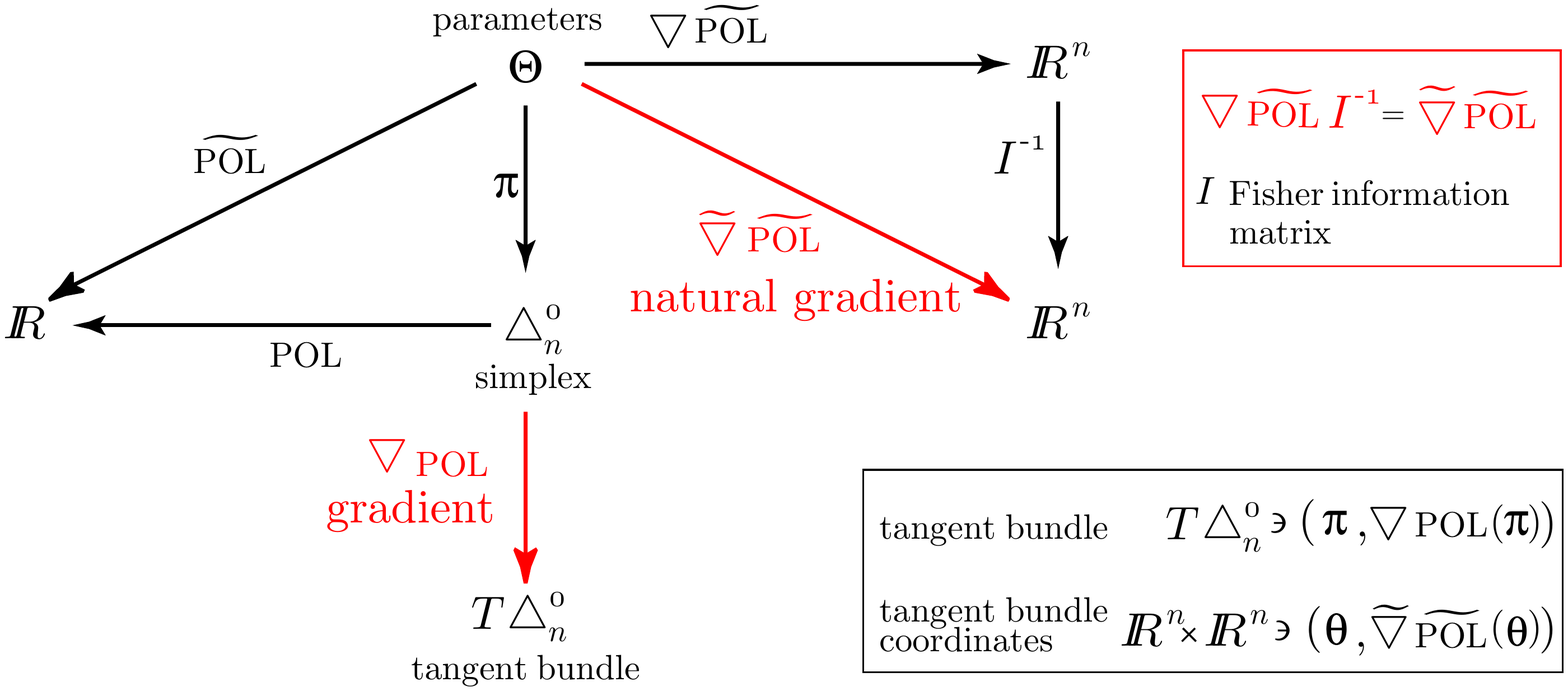}
\caption{Diagram of the action of the natural gradient in a given parametrization $\pi \colon \Theta \to \Delta_n^\circ$.\label{fig:albero}}
\end{figure}

The \emph{common parametrization} of the (flat) simplex $\Delta_n^\circ$ is the projection on the solid simplex $\Gamma_n = \setof{\bm\eta \in \reals^n}{0<\eta_j,  \sum_{j=1}^n \eta_j< 1}$, that is
\begin{equation*}
  \pi \colon \Gamma_n \ni \bm\eta \mapsto \left(1 - \sum_{j=1}^n \eta_j,\eta_1,\dots,\eta_n\right) \in \Delta_n^\circ,
\end{equation*}
in which case $\partial_j \pi(\bm\eta)$, $j=1,\dots,n$, is the random variable with values $-1$ at $x=0$, 1 at $x=j$, 0 otherwise, hence $\partial_j \pi(\bm \eta) = \left((X=j)-(X=0)\right)$ and
\begin{equation*}
  D_j \pi(\bm \eta) = \left((X=j)-(X=0)\right)/\pi(\bm\eta).
\end{equation*}

The element $(j,h)$ of the Fisher information matrix is
\begin{multline*}
  I_{jh}(\bm\eta)= \expectat {\pi(\bm\eta)}{\frac{(X=j)-(X=0)}{\pi(X;\bm\eta)}\frac{(X=h)-(X=0)}{\pi(X;\bm\eta)}} = \\  \sum_x \pi(x,\bm\eta)^{-1}\left((x=j)(j=h)+(x = 0)\right) = \eta_j^{-1}(j=h) + \left(1 - \sum_k \eta_k\right)^{-1}
\end{multline*}
hence
\begin{equation*}
  I(\bm\eta) = \diagof{\bm\eta}^{-1}  + \left(1 - \sum_{j=1}^n \eta_j\right)^{-1}[1]_{i,j=1}^n.
\end{equation*}

As an example we consider $n=3$.  The Fisher information matrix, its inverse and the determinant of the inverse are, respectively,
\begin{multline*}
 I(\eta_1,\eta_2,\eta_3)  = \\ (1 - \eta_1 - \eta_2-\eta_3)^{-1}
  \begin{bmatrix}
\eta_1^{-1}(1 - \eta_2-\eta_3) & 1 &1 \\ 1 & \eta_2^{-1}(1 - \eta_1-\eta_3) & 1 \\  1 &1& \eta_3^{-1}(1 - \eta_1-\eta_2)
  \end{bmatrix},
\end{multline*}
\begin{equation*}
I(\eta_1,\eta_2,\eta_3)^{-1}  =
  \begin{bmatrix}
(1 - \eta_{1}) \eta_{1} & -\eta_{1} \eta_{2} & -\eta_{1} \eta_{3}
\\
-\eta_{1} \eta_{2} & (1 - \eta_{2}) \eta_{2} & -\eta_{2} \eta_{3}
\\
-\eta_{1} \eta_{3} &  -\eta_{2} \eta_{3} & (1 - \eta_{3}) \eta_{3}
  \end{bmatrix}, 
\end{equation*}
\begin{equation*}
\det\left(I(\eta_1,\eta_2,\eta_3)^{-1} \right) = (1-\eta_1- \eta_2- \eta_3)\eta_1 \eta_2 \eta_3.
\end{equation*}
Note that the computation of the inverse of $I(\bm\eta)$ is an application of the Sherman-Morrison formula and the computation of the determinant of $I(\bm\eta)^{-1}$ is an application of the matrix determinant lemma.

For general $n$, we have the following Proposition, whose interest stems from the definition of natural gradient, see Eq. \eqref{def:nat-grad}.

\begin{prop}\label{prop:I-1} \
\begin{enumerate}
\item The inverse of the Fisher information matrix is
\begin{equation*}
I(\bm\eta)^{-1}=\diagof{\bm\eta}  - \bm\eta \bm\eta^t
\end{equation*}
\item In particular,  $I(\bm\eta)^{-1}$ is zero  on the vertexes of the simplex, only.
\item The determinant of the Fisher information matrix is
\begin{equation*}
\det\left(I(\bm\eta)^{-1}\right)=\left(1- \sum_{i=1}^n  \eta_i\right)  \prod_{i=1}^n \eta_i .
\end{equation*}
\item The determinant of $I(\bm\eta)^{-1}$  is zero on the borders of the simplex, only.
\item On the interior of each  facet, the rank of  $I(\bm\eta)^{-1}$ is $n-1$ and the $n-1$ liner independent column vectors  generate the subspace parallel to the facet itself.
\end{enumerate}
\end{prop}

\begin{proof} \
\begin{enumerate}
\item By direct computation, $I(\bm\eta) I(\bm\eta)^{-1}$ is the identity matrix.

\item The diagonal elements of  $I(\bm\eta)^{-1}$ are zero if $\eta_j=1$ or $\eta_j=0$, for $j=1,\dots,n$. If, for a given $j$,  $\eta_j=1$, then the elements of $I(\bm\eta)^{-1}$ are zero if $\eta_h=0$, $h \ne j$. The remaining case corresponds to  $\eta_j=0$  for all $j$. Then $I(\bm\eta)^{-1}=0$  on all the vertexes of the simplex.

\item It follows from Matrix Determinant Lemma.

\item The determinant factors in terms corresponding to the equations of the facets.
\item
Given $i$, the conditions  $\eta_i=0$ and $\eta_j\ne 0,1$ for all $j\ne i$, define the interior of the facet orthogonal to standard base vector $e_i$. In this case the $i$-th row and the $i$-th column of $ I(\bm\eta)^{-1}$  are zero and  the complement matrix corresponds to the inverse of a Fisher information matrix in dimension $n-1$ with non zero determinant. It follows that the subspace generated by the columns has dimension $n-1$ and coincides with the space orthogonal to $\eta_i$.
Consider  the facet defined by $\left(1- \sum_{i=1}^n \eta_i\right)=0$, $\eta_i\ne 0,1$ for all $i$.  For a given $j$,  the matrix without the $j$-th row and the $j$-th column has determinant
$\left(1-\sum_{i=1,i \ne j }^n \eta_i \right)\prod_{i=1, i \ne j}^n \eta_i$. On the considered facet this determinant is different to zero and $I(\bm\eta)^{-1}$  has rank $n-1$ and their columns are orthogonal to the constant vector.\qed

\end{enumerate}

\end{proof}

An other parametrization is the \emph{exponential parametrization} based on the exponential family with sufficient statistics $X_j = (X=j)$, $j=1,\dots,n$,
\begin{equation*}
  \pi \colon \mathbb R^n \ni \bm\theta \mapsto \expof{\sum_{j=1}^n \theta_j X_j - \psi(\bm\theta)}\frac1{n+1}
\end{equation*}
where
\begin{equation*}
\psi(\bm\theta) = \logof{1 + \sum_j \euler^{\theta_j}} - \logof{n+1}.
\end{equation*}
Some of the properties discussed in Prop. \ref{prop:I-1} should actually be discussed under the exponential parametrization, see e.g. \cite{malago|pistone:2014Entropy}, but we do not do that here. We will discuss the exponential parametrization below in Sec. \ref{sec:POLasE} to show that polarization can be seen as an expectation with respect to an exponential family.

\section{The gradient flow of $\POL$}
\label{sec:pol}

We apply now the general theory of the natural gradient to the study of the dynamics of the polarization measure.  Our goal is to find the lines of the steepest ascent of the function $\POL$.

In the common parametrization we have
\begin{equation*}
  \widetilde\POL(\bm\eta) = \left(1-\sum_{j=1}^n \eta_j\right)^2\left(\sum_{j=1}^n \eta_j\right)+\sum_{j=1}^n \eta_j^2\left(1-\eta_j\right)
\end{equation*}
and, for $n=2$,
\begin{multline*}
  \widetilde\POL(\bm\eta) = (1-\eta_1-\eta_2)^2(\eta_1+\eta_2) +\eta_1^2(1-\eta_1)+\eta_2^2(1-\eta_2)=\\
 3(\eta_1^2\eta_2+\eta_1\eta_2^2) -(\eta_1^2+\eta_2^2) -4\eta_1\eta_2+(\eta_1+\eta_2)
\end{multline*}
with gradient
\begin{equation*}
  \nabla\widetilde\POL(\bm\eta) =  \left(6\eta_{1} \eta_{2} + 3 \eta_{2}^{2} - 2 \eta_{1}- 4\eta_{2} + 1 ,
 6\eta_{2} \eta_{1} +  3 \eta_{1}^{2}- 2 \eta_{2}- 4\eta_{1}+1 \right).
\end{equation*}
The inverse of the Fisher information matrix is
\begin{equation*}
I^{-1}(\eta_1,\eta_2)  =
  \begin{bmatrix}
(1 - \eta_{1}) \eta_{1} & -\eta_{1} \eta_{2}
\\
-\eta_{1} \eta_{2} & (1 - \eta_{2}) \eta_{2}
  \end{bmatrix}.
\end{equation*}
The natural gradient is
\begin{multline}\label{eq:grapol2}
  \natnabla \widetilde\POL(\eta_1,\eta_2) = \nabla \widetilde \POL(\eta_1,\eta_2) I^{-1}(\eta_1,\eta_2) = \\
   \left(-9 \eta_{1}^{3}
\eta_{2} - 9 \eta_{1}^{2} \eta_{2}^{2} + 2 \eta_{1}^{3} +
14 \eta_{1}^{2} \eta_{2} + 5 \eta_{1} \eta_{2}^{2} - 3
\eta_{1}^{2} - 5 \eta_{1} \eta_{2} + \eta_{1},\right.\\ \left. -9 \eta_{1}^{2}
\eta_{2}^{2} - 9 \eta_{1} \eta_{2}^{3} + 5 \eta_{1}^{2}
\eta_{2} + 14 \eta_{1} \eta_{2}^{2} + 2 \eta_{2}^{3} - 5
\, \eta_{1} \eta_{2} - 3 \eta_{2}^{2} + \eta_{2} \right) .
\end{multline}
See in Fig.s \ref{fig:gradpol-A} and \ref{fig:gradpol-B} the gradient fields.

\begin{figure}
\begin{center}
\begin{tabular}{cc} 
\includegraphics*[width=.30\textwidth, viewport = 102 97 431 425]{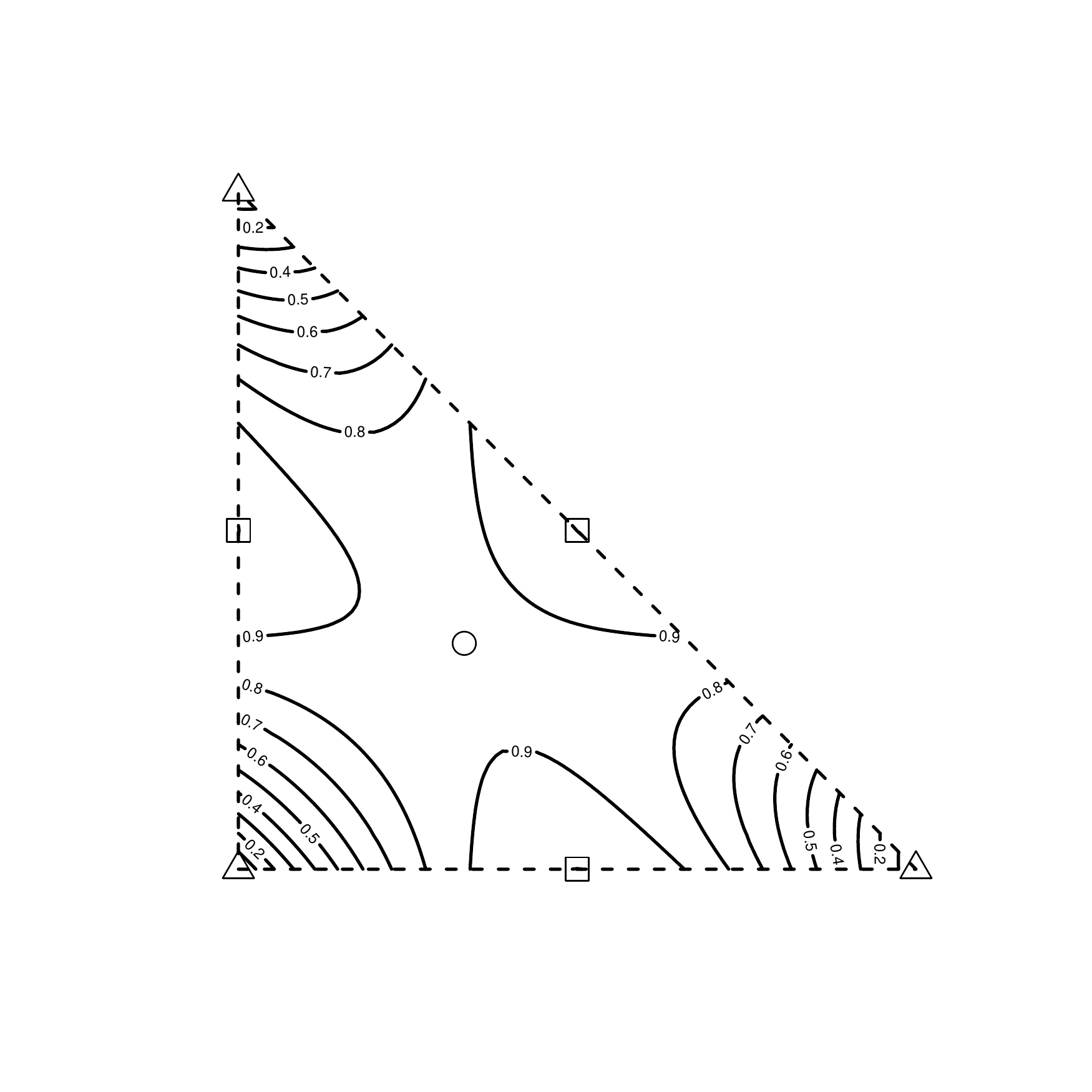} &
\includegraphics*[width=.30\textwidth, viewport = 94 86 439 432]{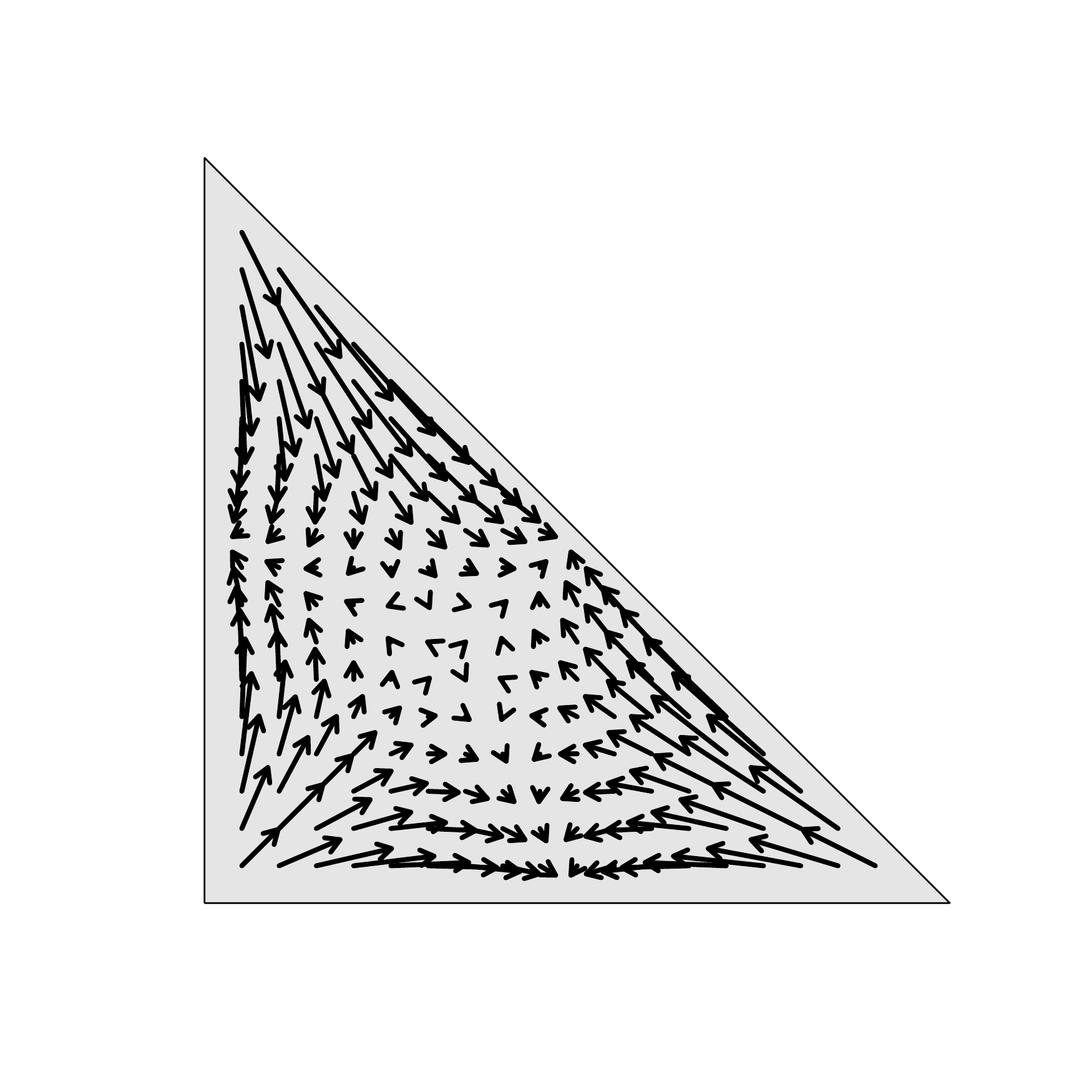}
\end{tabular}
\end{center}
\caption{Level curves of the Polarization in the common  parametrization (left) and natural gradient field (right).\label{fig:gradpol-A}}
\end{figure}

Some properties of the natural gradient field depend on the inverse of Fisher information matrix only, other are specific properties of the function $\POL$. We note that the vector field in \eqref{eq:grapol2} is actually defined and continuous for all $\bm \eta \in \reals^2$ and coincides with the natural gradient in the interior $\Gamma_n$ of the solid simplex. By abuse of language, we call the extended object with the same name of the probabilistic object.

The inverse of Fisher information matrix is zero at the points $(0,0)$, $(0,1)$, $(1,0)$, see Prop. \ref{prop:I-1}. These are among of the fixed points of the gradient flow equation.
The determinant of the inverse of Fisher information matrix is $\det I(\bm\eta) ^{-1} = (1-\eta_1-\eta_2) \eta_1\eta_2$.
As proved in  Prop. \ref{prop:I-1} for the general case, the determinant is zero on the borders of simplex. Here the probabilistic model is not defined, but the continuous extension of the gradient flow holds. On the facets of the simplex the vector field is parallel to the facets itself.
On the facets the inverse of Fisher information matrix is one dimensional:  if $\eta_1=0$ then $I^{-1}$ corresponds to $(0,1)^t$, if $\eta_2=0$ then $I^{-1}$ corresponds to $(1,0)^t$,  and if $1-\eta_1-\eta_2=0$ then $I^{-1}$ corresponds to $(1,-1)^t$.

\begin{figure}
\begin{center}
\begin{tabular}{cc} 
\includegraphics*[width=.30\textwidth, viewport = 94 86 439 432]{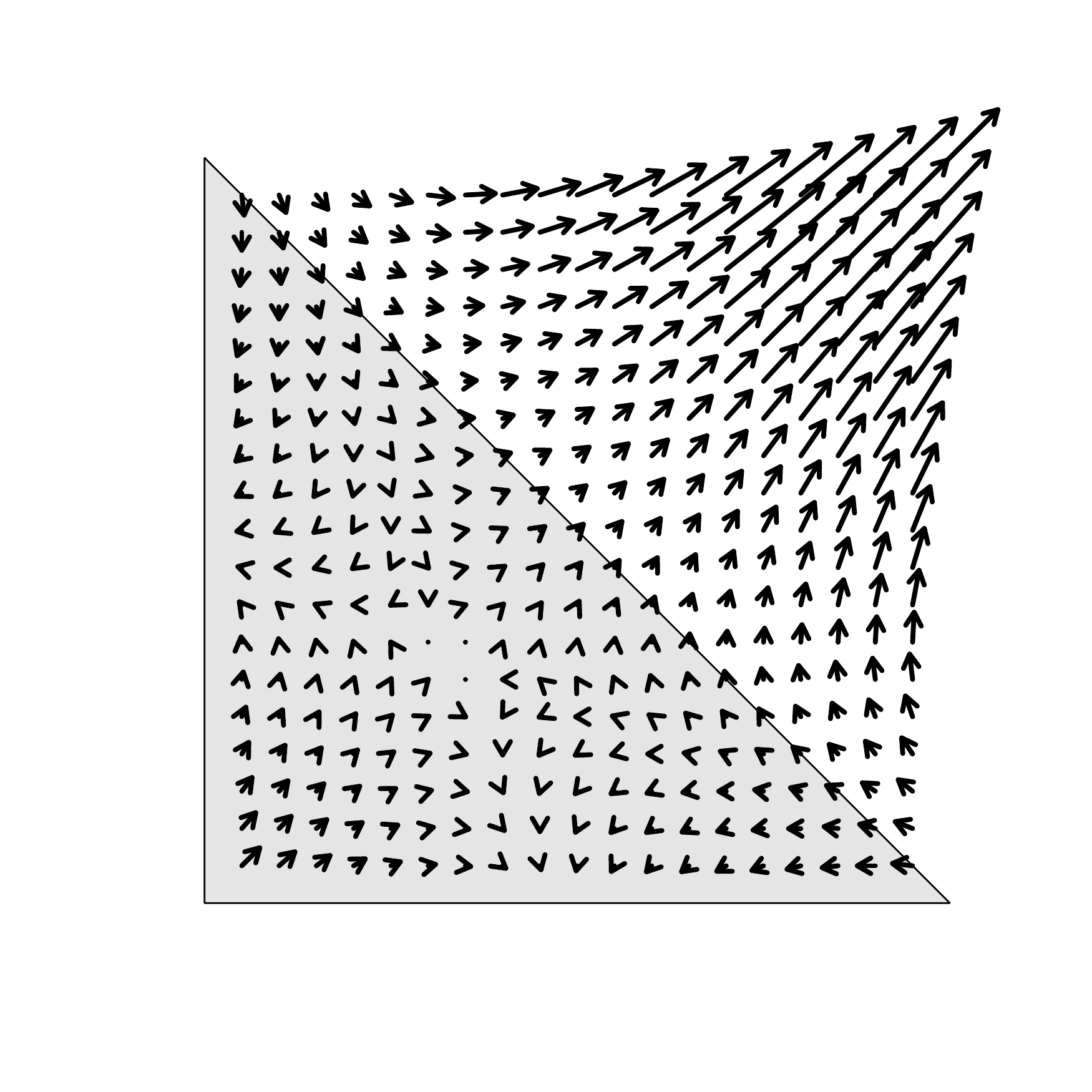} & 
\includegraphics*[width=.30\textwidth, viewport = 94 86 439 432]{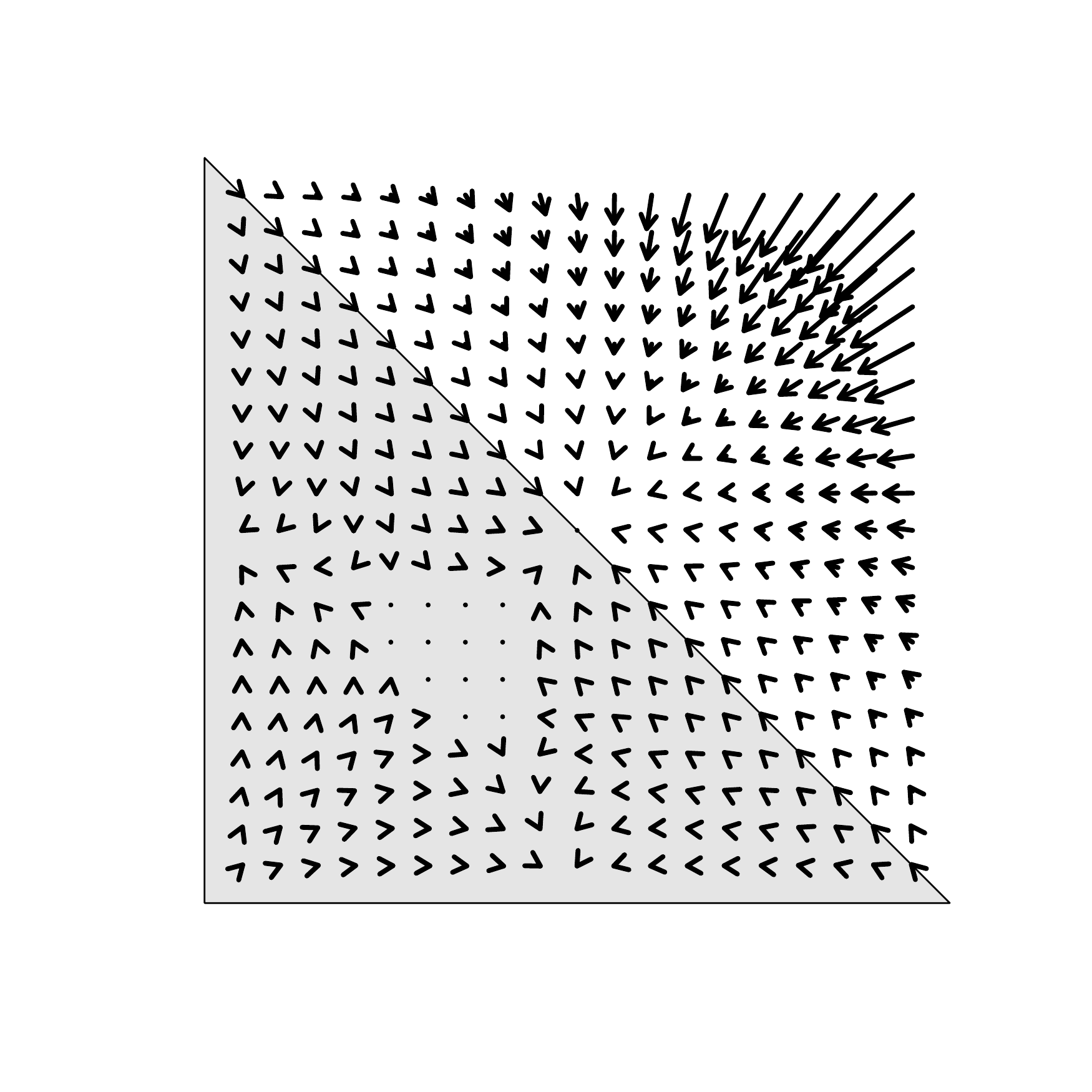}
\end{tabular}
\end{center}
\caption{The gradient without the correction by $I^{-1}$ gives the wrong directions (left), while the natural gradient applies correctly (right). Both fields are extended outside the probability simplex. The length of the arrows is relative to each display and cannot be compared across displays.\label{fig:gradpol-B}}
\end{figure}

To study the flow in the fixed points we consider the sign of the eigenvalues of the Jacobian of the natural gradient, calculated in the fixed points, see \citet{arnold:2006ODE}.

The Jacobian of $\natnabla \widetilde\POL(\bm\eta)$ is
$J = \left[ \frac{\partial \natnabla \widetilde\POL(\bm\eta)}{\partial \eta_1} \quad  \frac{\partial \natnabla \widetilde\POL(\bm\eta)}{\partial \eta_2}\right]$ where
\begin{equation*}
\frac{\partial \natnabla \widetilde\POL(\bm\eta)}{\partial \eta_1}=
\begin{bmatrix}
-27 \eta_1^2\eta_2 -18\eta_1\eta_2^2+6\eta_1^2+28\eta_1\eta_2+5\eta_2^2-6\eta_1-5\eta_2+1\\ -9\eta_2^3-18\eta_2^2\eta_1+14\eta_2^2+10\eta_2\eta_1-5\eta_2
\end{bmatrix}
\end{equation*}
\begin{equation*}
 \frac{\partial \natnabla \widetilde\POL(\bm\eta)}{\partial \eta_2}=
\begin{bmatrix}
-9\eta_1^3-18\eta_1^2\eta_2+14\eta_1^2+10\eta_1\eta_2-5\eta_1\\
-27 \eta_2^2\eta_1 -18\eta_2\eta_1^2+6\eta_2^2+28\eta_2\eta_1+5\eta_1^2-6\eta_2-5\eta_1+1
\end{bmatrix}
\end{equation*}

The Jacobian calculated in the vertexes are
\begin{equation*}
J(0,0)=J(0,1)=J(1,0)=\begin{bmatrix} 1 & 0 \\ 0 & 1\end{bmatrix}
\end{equation*}
and the two eigenvalues of the three Jacobian are both positive. Then, the fixed points on the vertexes repel flow locally.

Moreover, the natural gradient is zero on the midpoints of the borders, $\left(0,  \frac 1 2\right)$, $\left( \frac 1 2 ,0\right)$ and $\left( \frac 1 2,  \frac 1 2\right)$.
The Jacobian calculated in the midpoints are
\begin{equation*}
J\left(0,\frac 1 2\right)=-\frac 1 8\begin{bmatrix}  2  & 0 \\ 1 & 4\end{bmatrix} \quad
J\left(\frac 1 2,0\right)=-\frac 1 8\begin{bmatrix} 4 & 1 \\ 0 & 2\end{bmatrix} \quad
J\left(\frac 1 2,\frac 1 2\right)=-\frac 1 8\begin{bmatrix} 3  & -1 \\  -1  &  3 \end{bmatrix}
\end{equation*}
and the two eigenvalues of the three Jacobian are both negative. Then, these fixed points attract  flow locally.

In standard cases the Jacobian of a gradient is the Hessian matrix. This is not true anymore in IG where we compute the Jacobian of the natural gradient. Actually, in IG there are various notions of Hessian, each one based on a different connection on the tangent bundle. 

\section{Generalisation of the polarization measure}
\label{sec:POLs}
We consider now an inverse problem in the simple case $n=2$. We want a third degree symmetric polynomial which could be used as polarization measure. Computations were performed using Sage \cite{sage:2014}.

A generic polynomial is
\begin{multline*}
f(\bm \pi) = a (\pi_{0}^{3} +\pi_{1}^{3}+ \pi_{2}^{3})+ b (\pi_{0}^{2} \pi_{1} +  \pi_{0} \pi_{1}^{2}
 +  \pi_{0}^{2} \pi_{2} + \pi_{0} \pi_{2}^{2} +\pi_{1}^{2} \pi_{2} +  \pi_{1}\pi_{2}^{2})+\\
  c \pi_{0} \pi_{1} \pi_{2} + d (\pi_{0}^{2} +\pi_{1}^{2} +\pi_{2}^{2})
 + e (\pi_{0} \pi_{1} + \pi_{0} \pi_{2} +  \pi_{1} \pi_{2} )
,
\end{multline*}
and its expression in the parameters is
\begin{multline*}
  \tilde f(\bm \eta) = (-3 a + 3 b -  c) \eta_{1}^{2} \eta_{2} + (-3 a + 3 b -  c) \eta_{1}
\eta_{2}^{2} + (3 a -  b + 2 d -  e) \eta_{1}^{2} + \\ (6 a - 4 b + c + 2 d
-  e) \eta_{1} \eta_{2} + (3 a -  b + 2 d -  e) \eta_{2}^{2} +\\ (-3 a + b
- 2 d + e) \eta_{1} + (-3 a + b - 2 d + e) \eta_{2} + a + d.
\end{multline*}

The components of the natural gradient are
\begin{align*}
  (\natnabla \tilde f)_1(\bm \eta) = &(9 a - 9 b + 3 c) \eta_{1}^{3} \eta_{2} +
  (9 a - 9 b + 3 c) \eta_{1}^{2}\eta_{2}^{2} +
  (-6 a + 2 b - 4 d + 2 e) \eta_{1}^{3} + \\
  &(-18 a + 14 b - 4c - 4 d + 2 e) \eta_{1}^{2} \eta_{2} +
  (-9 a + 5 b -  c - 4 d + 2 e)\eta_{1} \eta_{2}^{2} + \\
  &(9 a - 3 b + 6 d - 3 e) \eta_{1}^{2} +
  (9 a - 5b + c + 4 d - 2 e) \eta_{1} \eta_{2} +\\
  &(-3 a + b - 2 d + e) \eta_{1} \\
  (\natnabla \tilde f)_2(\bm \eta) = & (9 a - 9 b + 3 c) \eta_{1}\eta_{2}^{3} +
  (9 a - 9 b + 3 c) \eta_{1}^{2} \eta_{2}^{2} +
  (-6 a + 2 b - 4 d + 2 e) \eta_{2}^{3}+\\
  &(-18 a + 14 b - 4 c - 4 d + 2 e) \eta_{1} \eta_{2}^{2}
  (-9 a + 5 b -  c - 4 d + 2 e) \eta_{1}^{2} \eta_{2} + \\
 &(9 a - 3 b + 6 d - 3 e) \eta_{2}^{2} +
 (9 a - 5 b + c + 4 d - 2 e) \eta_{1} \eta_{2} +  \\
 &(-3 a + b - 2 d + e) \eta_{2}.
\end{align*}

Note that the case $\tilde f = \widetilde\POL$ is $b = 1$ and $a = c = d = e = 0$. Because of the multiplication by $I(\bm \eta)^{-1}$, the natural gradient $\natnabla \tilde f(\bm \eta)$ is zero in each of the simplex vertexes $(0,0)$, $(1,0)$, $(0,1)$. Because of the symmetry, the natural gradient $\natnabla \tilde f(\bm \eta)$ is zero in each of the mid-point of the edges, i.e. $(1/2,0)$, $(0,1/2)$, $(1/2,1/2)$ and at the uniform probability, $(1/3,1/3)$. We need to turn to the discussion of the Jacobian matrix of the natural gradient of $f$.

At the vertexes, the Jacobian is
\begin{equation*}
  ( -3 a + b - 2 d + e )
\begin{bmatrix}
   1 & 0 \\
0 & 1
  \end{bmatrix},
\end{equation*}
so that we want $-3 a + b - 2 d + e > 0$.

At the uniform probability $(1/3,1/3)$ the Jacobian is
\begin{equation*}
\frac 1 9  (6 a -  c + 6 - 3 e) \begin{bmatrix}
   1  & 0 \\
0 & 1
  \end{bmatrix}.
\end{equation*}
A \emph{necessary} condition of non definiteness is
\begin{equation}\label{eq:nec-H}
6 a -  c + 6 d - 3 e = 0.
\end{equation}

At the mid-points of the three edges $(1/2,0)$, $(0,1/2)$, $(1/2,1/2)$, the values of the Jacobian matrices are, respectively,
\begin{align*}
 & \begin{bmatrix}
    \frac{3}{2} a - \frac{1}{2} b + d - \frac{1}{2} e & \frac{9}{8} a -
\frac{1}{8} b - \frac{1}{8} c + d - \frac{1}{2} e \\
0 & -\frac{3}{4} a - \frac{1}{4} b + \frac{1}{4} c -  d +
\frac{1}{2} e
  \end{bmatrix}
\\
  &\begin{bmatrix}
-\frac{3}{4} a - \frac{1}{4} b + \frac{1}{4} c -  d + \frac{1}{2} e & 0 \\
\frac{9}{8} a - \frac{1}{8} b - \frac{1}{8} c + d - \frac{1}{2} e & \frac{3}{2} a - \frac{1}{2} b + d - \frac{1}{2} e
  \end{bmatrix}
\\
  &\begin{bmatrix}
\frac{3}{8} a - \frac{3}{8} b + \frac{1}{8} c & -\frac{9}{8} a + \frac{1}{8} b + \frac{1}{8} c -  d + \frac{1}{2} e \\
-\frac{9}{8} a + \frac{1}{8} b + \frac{1}{8} c -  d + \frac{1}{2} e & \frac{3}{8} a - \frac{3}{8} b + \frac{1}{8} c
  \end{bmatrix} ,
\end{align*}
and, imposing the necessary condition  $c=6a+6d-3e$ of Eq. \eqref{eq:nec-H}, they are
\begin{equation*}
  \frac 1 8 (3a-b+2d-e)\begin{bmatrix}
    4 & 1 \\ 0 & 2
  \end{bmatrix},
\quad
   \frac 1 8 (3a-b+2d-e)\begin{bmatrix}
2 & 0 \\ 1 & 4
  \end{bmatrix},
\quad
  \frac 1 8 (3a-b+2d-e)\begin{bmatrix}
3 & -1 \\ -1 &3
  \end{bmatrix}.
\end{equation*}
The two eigenvalues of each of these matrices are both negative if  $3a-b+2d-e<0$. Notice that this condition of attracting flow on the midpoints is the same condition of repelling flow on the vertexes.

Summarising, the conditions on the coefficients of a third degree symmetric polynomial that represent a measure with the properties stated above, are
\begin{equation*}
6 a -  c + 6d - 3 e = 0 \qquad 3a-b+2d-e<0 .
\end{equation*}

In conclusion, it is possible to design other measures of polarization in the form of a symmetric polynomial of degree 3, or more general forms. In particular, it would be interesting to have a measure in a form similar to the entropy. The analogy is suggested by the fact that the gradient flow of the entropy gives trajectories that move from the uniform probability to one of the vertexes of the simplex.

\section{$\POL$ as expectation along an exponential family}
\label{sec:POLasE}
We have already observed that the function $\POL$ defined in Eq. \ref{eq:POL} is an homogeneous polynomial of degree 3 in the indeterminates $\pi_x$. The general class of indexes based on homogeneous polynomials that we have discussed in Sec. \ref{sec:POLs} is of special interest because they reduce to an expectation with respect to an exponential family, as we discuss now in the case of three sample points.

If the random variables $X,Y,Z$ are i.i.d. with the distribution of $X$ supported by $\set{0,1,2}$, and coded as an exponential family i.e.
\begin{equation*}
  \begin{array}{c|cc}
    X & X_1 & X_2 \cr
\hline
0 & 0 & 0 \cr
1 & 1 & 0 \cr
2 & 0 & 1 \cr
\end{array}
, \qquad
 \begin{array}{rl}
 \pi_x &= \expof{\theta_1 x_1 + \theta_2 x_2 - \psi(\theta_1,\theta_2)}, \\
\psi(\theta_1,\theta_2) &= \logof{1 + \euler^{\theta_1}+\euler^{\theta_2}},
 \end{array}
\end{equation*}
then the joint probability function of $X,Y,Z$ is itself an exponential family,
\begin{equation}\label{eq:3-exp}
  p(x,y,z;\theta_1,\theta_2) = \euler^{\theta_1 (x_1+y_1+z_1) + \theta_2 (x_2 + y_2 + z_2) -3 \psi(\theta_1,\theta_2)}, \quad x, y, z = 0,1,2,
\end{equation}
with sufficient statistics $T_j(x,y,z) = x_j+y_j+z_j$, $j=1,2$. The model for $T_1,T_2$ is
\begin{equation}\label{eq:2-exp}
  p(t_1,t_2;\theta_1,\theta_2) = \euler^{\theta_1 t_1 + \theta_2 t_2 -3 \psi(\theta_1,\theta_2)} f(t_1,t_2),
\end{equation}
where $f$ is the table count of $T_1,T_2$. The random variables of this model are represented in Tab. \ref{tab:1}.
\begin{table}
  \caption{Sample space and random variables of the exponential family in Eq. \eqref{eq:3-exp}. The sample cases of polarization are presented in boldface in the first column.\label{tab:1}}\scriptsize
\begin{equation*}
\begin{array}{rrrrrrrrrrrr}
 & X & Y & Z & X_1 & Y_1 & Z_1 & X_2 & Y_2 & Z_2 & T_1 & T_2 \\
  \cline{2-12}
1 & 0 & 0 & 0 & 0 & 0 & 0 & 0 & 0 & 0 & 0 & 0 \\
\bm  2 & 1 & 0 & 0 & 1 & 0 & 0 & 0 & 0 & 0 & 1 & 0 \\
\bm   3 & 2 & 0 & 0 & 0 & 0 & 0 & 1 & 0 & 0 & 0 & 1 \\
\bm 4 & 0 & 1 & 0 & 0 & 1 & 0 & 0 & 0 & 0 & 1 & 0 \\
\bm  5 & 1 & 1 & 0 & 1 & 1 & 0 & 0 & 0 & 0 & 2 & 0 \\
  6 & 2 & 1 & 0 & 0 & 1 & 0 & 1 & 0 & 0 & 1 & 1 \\
\bm  7 & 0 & 2 & 0 & 0 & 0 & 0 & 0 & 1 & 0 & 0 & 1 \\
  8 & 1 & 2 & 0 & 1 & 0 & 0 & 0 & 1 & 0 & 1 & 1 \\
\bm  9 & 2 & 2 & 0 & 0 & 0 & 0 & 1 & 1 & 0 & 0 & 2 \\
\bm{10} & 0 & 0 & 1 & 0 & 0 & 1 & 0 & 0 & 0 & 1 & 0 \\
\bm{11} & 1 & 0 & 1 & 1 & 0 & 1 & 0 & 0 & 0 & 2 & 0 \\
  12 & 2 & 0 & 1 & 0 & 0 & 1 & 1 & 0 & 0 & 1 & 1 \\
\bm{13} & 0 & 1 & 1 & 0 & 1 & 1 & 0 & 0 & 0 & 2 & 0 \\
  14 & 1 & 1 & 1 & 1 & 1 & 1 & 0 & 0 & 0 & 3 & 0 \\
\bm{15} & 2 & 1 & 1 & 0 & 1 & 1 & 1 & 0 & 0 & 2 & 1 \\
  16 & 0 & 2 & 1 & 0 & 0 & 1 & 0 & 1 & 0 & 1 & 1 \\
\bm{17} & 1 & 2 & 1 & 1 & 0 & 1 & 0 & 1 & 0 & 2 & 1 \\
\bm{18} & 2 & 2 & 1 & 0 & 0 & 1 & 1 & 1 & 0 & 1 & 2 \\
\bm{19} & 0 & 0 & 2 & 0 & 0 & 0 & 0 & 0 & 1 & 0 & 1 \\
  20 & 1 & 0 & 2 & 1 & 0 & 0 & 0 & 0 & 1 & 1 & 1 \\
\bm{21} & 2 & 0 & 2 & 0 & 0 & 0 & 1 & 0 & 1 & 0 & 2 \\
  22 & 0 & 1 & 2 & 0 & 1 & 0 & 0 & 0 & 1 & 1 & 1 \\
\bm{23} & 1 & 1 & 2 & 1 & 1 & 0 & 0 & 0 & 1 & 2 & 1 \\
\bm{24} & 2 & 1 & 2 & 0 & 1 & 0 & 1 & 0 & 1 & 1 & 2 \\
 \bm{25} & 0 & 2 & 2 & 0 & 0 & 0 & 0 & 1 & 1 & 0 & 2 \\
\bm{26} & 1 & 2 & 2 & 1 & 0 & 0 & 0 & 1 & 1 & 1 & 2 \\
  27 & 2 & 2 & 2 & 0 & 0 & 0 & 1 & 1 & 1 & 0 & 3 \\
   \cline{2-12}
\end{array}
\end{equation*}
\end{table}

The marginal polytope, that is the convex set generated by the values of the sufficient statistics, is illustrated in Fig. \ref{fig:Mpolytope}. We refer to \cite{brown:86} for the relevant theory.
\begin{figure}
\begin{center}
\begin{tabular}{m{.65\textwidth}m{.30\textwidth}}
\centering
\includegraphics[width=.5\linewidth, viewport = 3 14 480 456]{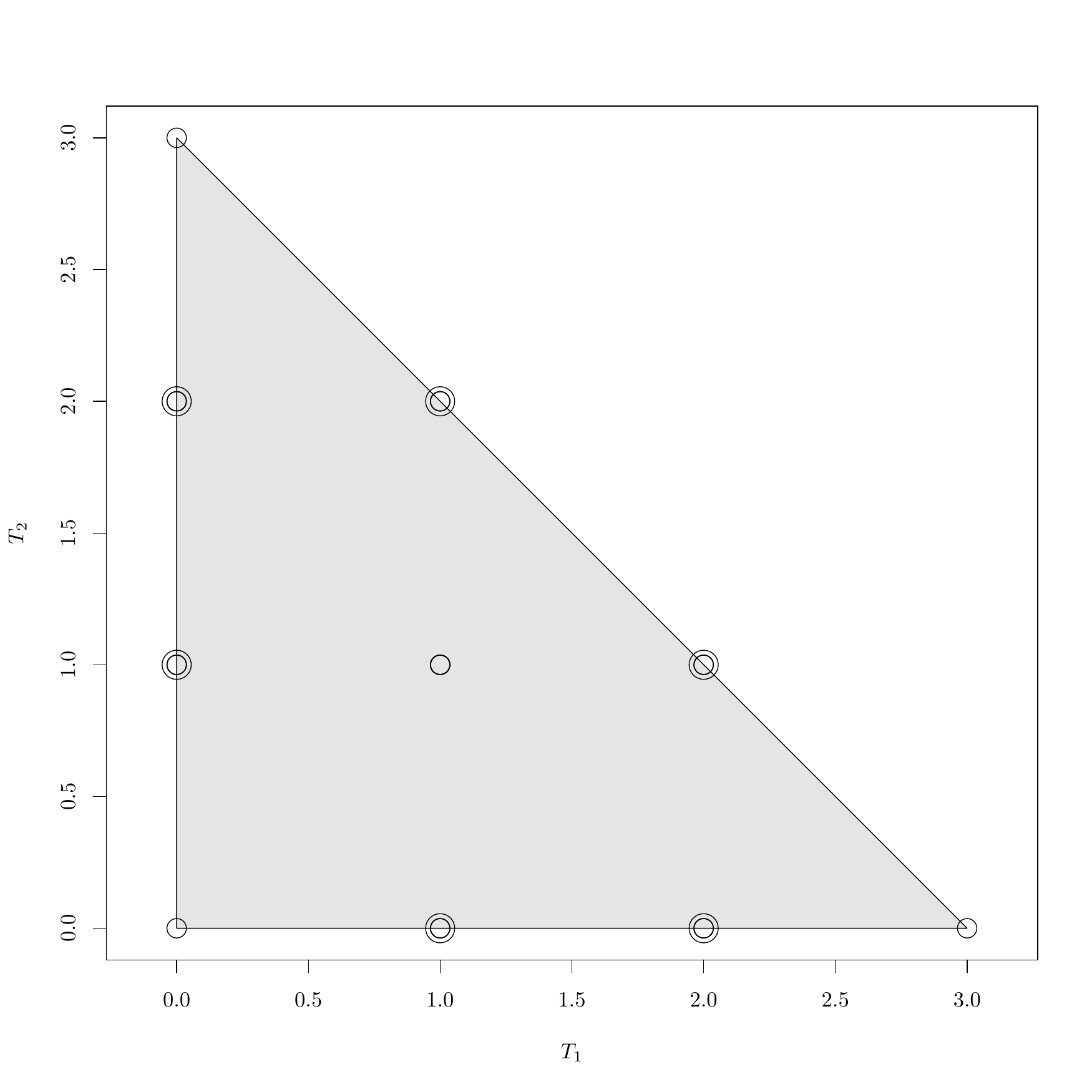}
&
$\begin{array}{rrrrr}
 T_1 \backslash T_2 & 0 & 1 & 2 & 3 \\
  \cline{2-5}
0 &   1 &   3 &   3 &   1 \\
  1 &   3 &   6 &   3 &   0 \\
  2 &   3 &   3 &   0 &   0 \\
  3 &  1 &   0 &   0 &   0 \\
   \cline{2-5}
\end{array}
$
\end{tabular}
\end{center}
\caption{Marginal polytope of the exponential family. The points in the diagram are the values of the sufficient statistics. Each one of the emphasised points correspond to three cases where polarization occurs, for example $T_1=1,T_2=2$ corresponds to the cases $\bm{18}: (2,2,1);\bm{24}: (2,1,2);\bm{26}: (1,2,2)$, in Tab. \ref{tab:1}.  At the right, the table of joint counts of $T_1$, $T_2$ is displayed.\label{fig:Mpolytope}}
\end{figure}

If $I$ is the indicator function of the set $\set{(0,1),(0,2),(1,0),(1,2),(2,0),(2,1)}$ then the value of $3 \POL$, when expressed in the parameters $\bm\theta$, is $\expectat {\bm\theta} {I(T_1,T_2)}$. Thus, the problem of the maximisation of the polarization measure is rephrased to the problem of finding the maximum of the expected value of a random variable on a given exponential family. This approach has been considered in Combinatorial Optimisation, see e.g. \cite{malago:2012thesis,malago|matteucci|pistone:2011CEC,arnoldetal:2011arXiv,malago|matteucci|pistone:2013CEC}. It has a number of issues.

 First, any convergent evolution along the exponential family has either a limit internal to the exponential family itself of a limit outside the exponential family, supported by a face of the marginal polytope, see \cite{cencov:72,rinaldo|feinberg|Zhou:2009,rauh|kahle|ay:2011,malago|pistone:arXiv1012.0637}.

  Second, the maximum of the function can be reached as a limit of expected values if, and only if, there exists, among the distributions obtained conditioning the exponential family to one face of the marginal polytope, a distribution such that the expected value of the function is equal to the maximum of the function itself. We do not enter here in a detailed discussion, see more information in \cite{malago|pistone:arXiv1012.0637,malago|pistone:2014Entropy}.

The expectation parameters of the exponential family \eqref{eq:3-exp} are
\begin{align*}
  \expectat {\bm \theta} {T_1} &= 3 \frac{\partial}{\partial\theta_1} \psi(\theta_1,\theta_2) = 3 \frac {\euler^{\theta_1}}{1 + \euler^{\theta_1} + \euler^{\theta_2}}, \\
  \expectat {\bm \theta} {T_2}&= 3 \frac{\partial}{\partial\theta_2} \psi(\theta_1,\theta_2) = 3 \frac {\euler^{\theta_2}}{1 + \euler^{\theta_1} + \euler^{\theta_2}}.
\end{align*}
These parameters are related to the $\bm\eta$ by
\begin{equation*}
\expectat {\bm\theta}{T_j} = \expectat {\bm\theta}{X_j} + \expectat {\bm\theta}{Y_j} + \expectat {\bm\theta}{Z_j} = 3\eta_j, \quad j = 1,2.
\end{equation*}

The inverse of the map
\begin{equation*}
 \frac 13 \nabla \psi \colon \bm \theta \mapsto \bm\eta
\end{equation*}
can be computed explicitly as
\begin{equation*}
  \euler^{\theta_1} = \frac{\eta_{1}}{1 - \eta_{1} - \eta_{2}}, \quad
  \euler^{\theta_2} = \frac{\eta_{2}}{1 - \eta_{1} - \eta_{2}}.
\end{equation*}

The cumulant function is expressed as a function of $\bm\eta$ as
\begin{equation*}
  3\psi(\theta_1,\theta_2) = 3\logof{1 + \euler^{\theta_1}+\euler^{\theta_2}} = \log{\frac{1}{(1 - \eta_1-\eta_2)^3}},
\end{equation*}
and the exponential family in Eq. \eqref{eq:3-exp} is expressed in the parameter $\eta$ as
\begin{multline}\label{eq:toric}
  p_{\bm\theta}(t_1,t_2) = \euler^{\theta_1 t_1}\euler^{\theta_2 t_2}\euler^{-3\psi(\theta_1,\theta_2)} f(t_1,t_2) = \\
 \eta_1^{t_1} \eta_2^{t_2} (1 - \eta_1 - \eta_2)^{3 - t_1 - t_2} f(t_1,t_2) = p_{\bm\eta}(t_1,t_2).
\end{multline}

When the probabilities are expressed in the form of Eq. \eqref{eq:toric}, we can actually  compute values for border cases \citep{pistone:2009SL}. We recover in a different way the result already known.
\begin{description}
\item[$\eta_1=0,\eta_2=1/2$:] Eq. \eqref{eq:toric} becomes $p_{0,1/2}(t_1,t_2) = 0^{t_1}(1/2)^{3-t_2} f(t_1,t_2)$, which is zero but for $t_1 = 0$. This distribution is concentrated on $\set{(0,1),(0,1),(0,2),(0,3)}$, with $p_{0,1/2}(0,t_2) = f(0,t_2)/8$. It follows $\expectat {0,1/2} {I(T_1,T_2)} = 3/4$.
\item[$\eta_1=1/2,\eta_2=0$:] Same as the previous case. The distribution is concentrated where $t_2=0$, $p_{1/2,0}(t_1,0) = f(t_1,0)/8$, and $\expectat {1/2,0} {I(T_1,T_2)} = 3/4$.
\item[$\eta_1=1/2,\eta_2=1/2$:] In this case the distribution is $(1/2)^{t_1} (1/2)^{t_2} 0^{3-t_1-t_2} f(t_1,t_2)$ with support on the face where $t_1+t_2=3$, probabilities $p_{1/2,1/2}(t,3-t) = f(t,3-t)/8$, and $\expectat {1/2,1/2} {I(T_1,T_2)} = 3/4$.
\end{description}

We have shown on an example that the maximal value of the polarization measure is actually reachable as a limit of the expected value of a random variable $I(T_1,T_2)$ on the exponential family, but this does not mean that the maximum value of the expected value is equal to the maximum value of the random variable itself. This would be true if the maximum of the random variable were reached on a face of the marginal polytope. This is discussed in the literature we have cited above.



\section{Conclusion and suggested applications}
\label{sec:app}
We have considered a statistical index different from indexes for concentration or uniformity, such as the discrete Gini index or Boltzmann-Gibbs-Shannon entropy. The polarization measure has been discussed in a dynamic way, by considering its variation and computing the directions of steepest variation. The study requires tools suitable to discuss differential equation on a differentiable manifold.

This methodology suggests to implement the velocity of variation itself as a statistical index. Consider a study of the evolution of an index in time such as \cite{pino|vidal-robert:2013}. In the time series $\pi_1,\pi_2,\dots$ the evolution of the index e.g., $\polof{\pi_1}, \polof{\pi_2},\dots$, could be misleading, because an increase in the index could be associated to a shift from a basin of attraction to a different basin of attraction. We suggest a more precise local study as follows. Given a movement from $\pi_t$ to $\pi_{t+1}$, we look for a comparison of an estimate $\overrightarrow{\pi_t\pi_{t+1}}$ of the velocity vector to the gradient field of the index, that is compute $\scalarat{\pi_t}{\overrightarrow{\pi_t\pi_{t+1}}}{\nabla \polof{\pi_t}}$. An estimator of the velocity is a mapping from a couple of densities $\pi_{\text{initial}}$, $\pi_{\text{final}}$ to the tangent space at the initial density $T_{\pi_{\text{initial}}} \Delta_n^\circ$. The inverse of such a mapping is discussed under the name of \emph{retraction} in \cite{absil|mahony|sepulchre:2008}.  The simplest example here being $\overrightarrow{\pi_t\pi_{t+1}} = (\pi_{t+1}-\pi_{t})/\pi_{t} = \pi_{t+1}/\pi_t - 1$, which is suggested by
\begin{equation*}
D\pi(t) = \derivby t \log \pi(i) = \lim_{h\to 0} h^{-1} \frac{\pi_{t+h}-\pi_{t}}{\pi_{t}}.
\end{equation*}
It is a common practice in Engineering to use the initial velocity of the Riemannian geodesic connecting $\pi_t$ to $\pi_{t+1}$ \citep{absil|mahony|sepulchre:2008}. This would require the computation of the geodesic itself, which is done using the computations schetched in App. \ref{sec:hessian}.


\def\cprime{$'$}

\appendix

\section{From Lotka-Volterra to the replicator}
\label{sec:replicator}

\subsection{Lotka-Volterra $n=2$}

From \cite{goel|maitra|montroll:1971}, with $\alpha_i,\lambda_i > 0$, $i=1,2$, and $t \mapsto (N_1(t),N_2(t)) \in \reals_>^2$, the LV equation
\begin{equation}
  \label{eq:LV1}
  \left\{
    \begin{aligned}
      \dot N_1 &= \alpha_1 N_1 - \lambda_1 N_1N_2 = N_1(\alpha_1 - \lambda_1 N_2), \\
      \dot N_2 &= -\alpha_2 N_2 + \lambda_2 N_1N_2 = N_2(-\alpha_2 + \lambda_2 N_1).
    \end{aligned}
\right.
\end{equation}
has a stationary point $(q_1,q_2)$ with
\begin{equation}
q_1 = \frac{\alpha_2}{\lambda_2}, \quad q_2 = \frac{\alpha_1}{\lambda_1}.
\end{equation}
In the variables
\begin{equation}
\left\{\begin{aligned}
z_1 &= \frac{N_1}{q_1} = \frac{\lambda_2N_1}{\alpha_2}, \\
z_2 &=\frac{N_2}{q_2} = \frac{\lambda_1N_2}{\alpha_1},
     \end{aligned}\right.
   \end{equation}
the equations are
\begin{equation}
  \label{eq:LV2}
  \left\{
    \begin{aligned}
      \dot z_1 &= \alpha_1 z_1 (1-z_2), \\
      \dot z_2 &= \alpha_2 z_2 (z_1-1),
    \end{aligned}
\right.
\quad\text{i.e.}
\quad
  \left\{
    \begin{aligned}
      \derivby t \log z_1 &= \alpha_1 (1-z_2), \\
      \derivby t \log  z_2 &= \alpha_2 (z_1-1).
    \end{aligned}
\right.
\end{equation}

It follows that
\begin{equation}
  \alpha_2 \frac{1-z_1}{z_1} \dot z_1 + \alpha_1 \frac{1 - z_2}{z_2} \dot z_2 = 0,
\end{equation}
so that
\begin{equation}
 C(z_1,z_2) = \alpha_2(\log z_1 - z_1) + \alpha_1 (\log z_2 - z_2) = \logof{(z_1\euler^{-z_1})^{\alpha_2}(z_2\euler^{-z_2})^{\alpha_1}}
\end{equation}
is constant. Note that $\lim_{z_1,z_2\to 0,+\infty} C(z_1,z_2) = -\infty$ and
\begin{equation}
  \Hessian C(z_1,z_2) = -
  \begin{bmatrix}
    \dfrac{\alpha_2}{z_1^2} & 0 \\ 0 & \dfrac{\alpha_1}{z_2^2}.
  \end{bmatrix}, \quad \text{is negative definite.}
\end{equation}
This shows the existence of periodic orbits, see \cite[p 10--11]{goel|maitra|montroll:1971}.

\subsection{Uplift of Lotka-Volterra}

Because of the periodic behaviour, we cannot expect the dynamic project to the simplex $\Delta_1$. Following \cite{hofbauer:1981}, we can go up to
\begin{equation}
\label{eq:Delta2open}
  \Delta_2^\circ = \setof{\pi = (\pi_0,\pi_1,\pi_2)}{\sum_i \pi_i = 1, \pi_i > 0, i=0,1,2}.
\end{equation}

We add a constant population of one individual $z_0=1$ and define the transformation $\reals_>^2 \leftrightarrow \Delta_2^\circ$
\begin{equation}
  \left\{
  \begin{aligned}
    \pi_0 &= \frac{z_0}{z_0+z_1+z_2} &= \frac{1}{1+z_1+z_2} \\
    \pi_1 &= \frac{z_1}{z_0+z_1+z_2} &= \frac{z_1}{1+z_1+z_2} \\
    \pi_2 &= \frac{z_2}{z_0+z_1+z_2} &= \frac{z_2}{1+z_1+z_2} \\
  \end{aligned}
  \right. , \qquad
  \left\{
  \begin{aligned}
    z_1 &= \frac{\pi_1}{\pi_0} \\
    z_2 &= \frac{\pi_2}{\pi_0}
  \end{aligned}
  \right.
\end{equation}

Note that $\dot\pi_0+\dot\pi_1+\dot\pi_2 = \derivby t \left(\pi_0+\pi_oz_1 + \pi_0z_2\right) = 0$. We have
\begin{multline}
  \dot\pi_0 = \derivby t (1+z_1+z_2)^{-1} = - \pi_0^2(\dot z_1 + \dot z_2) = - \pi_0^2(\alpha_1 z_1(1-z_2) + \alpha_2 z_2 (z_1-1)) = \\ - \pi_0^2\left(\alpha_1 \frac{\pi_1}{\pi_0}\left(1-\frac{\pi_2}{\pi_0}\right) + \alpha_2 \frac{\pi_2}{\pi_0} \left(\frac{\pi_1}{\pi_0}-1\right)\right) = - \left(\alpha_1 \pi_1\left(\pi_0-\pi_2\right) + \alpha_2 \pi_2 \left(\pi_1-\pi_0\right)\right) = \\
\pi_0\left(-\alpha_1\pi_1+\alpha_2\pi_2 +(\alpha_1-\alpha_2)\frac{\pi_1\pi_2}{\pi_0}\right),
\end{multline}

\begin{multline}
  \dot\pi_1 = \derivby t (\pi_0z_1) = \dot\pi_0 z_1 + \pi_0\dot z_1 = \dot\pi_0\frac{\pi_1}{\pi_0} + \pi_0\alpha_1 \frac{\pi_1}{\pi_0}\left(1-\frac{\pi_2}{\pi_0}\right)\\
= \pi_1\left(-\alpha_1\pi_1+\alpha_2\pi_2 +(\alpha_1-\alpha_2)\frac{\pi_1\pi_2}{\pi_0}\right) + \alpha_1 \pi_1 \left(1-\frac{\pi_2}{\pi_0}\right) \\ =
\pi_1\left(-\alpha_1\pi_1+\alpha_2\pi_2 +(\alpha_1-\alpha_2)\frac{\pi_1\pi_2}{\pi_0} + \alpha_1 \left(1-\frac{\pi_2}{\pi_0}\right)\right),
\end{multline}

\begin{multline}
  \dot\pi_2 = \derivby t (\pi_0z_2) = \dot\pi_0 z_2 + \pi_0\dot z_2 = \dot\pi_0\frac{\pi_2}{\pi_0} + \pi_0\alpha_2 \frac{\pi_2}{\pi_0}\left(\frac{\pi_1}{\pi_0}-1\right)\\
= \pi_2\left(-\alpha_1\pi_1+\alpha_2\pi_2 +(\alpha_1-\alpha_2)\frac{\pi_1\pi_2}{\pi_0}\right) + \alpha_2 \pi_2 \left(\frac{\pi_1}{\pi_0}-1\right) \\ =
\pi_2\left(-\alpha_1\pi_1+\alpha_2\pi_2 +(\alpha_1-\alpha_2)\frac{\pi_1\pi_2}{\pi_0} + \alpha_2 \left(\frac{\pi_1}{\pi_0}-1\right)\right).
\end{multline}

Define
\begin{equation}
  \label{eq:rep1}
  f_0(\pi) = 0, \quad f_1(\pi)=\alpha_1\left(1-\frac{\pi_2}{\pi_0}\right), \quad f_2(\pi) = \alpha_2\left(\frac{\pi_1}{\pi_0} - 1\right).
\end{equation}
Then
\begin{multline}
\label{eq:rep2}
  \pi\cdot f(\pi) = \pi_0 f_0(\pi) + \pi_1 f_1(\pi) + \pi_2 f_2(\pi) = \\
\pi_1 \alpha_1\left(1-\frac{\pi_2}{\pi_0}\right) + \pi_2 \alpha_2\left(\frac{\pi_1}{\pi_0} - 1\right) = \\
 \alpha_1 \pi_1 - \alpha_2 \pi_2 + (\alpha_2 - \alpha_1) \frac{\pi_1\pi_2}{\pi_0},
\end{multline}
hence the differential equation for $\pi$ is the \emph{replicator equation}
\begin{equation}
  \label{eq:rep}
  \dot \pi_i = \pi_i(f_i(\pi) - \pi\cdot f(\pi)), \quad i = 1,2,3.
\end{equation}

\section{Information Geometry of the replicator in dimension 2}

The replicator equation Eq. \eqref{eq:rep} is nothing else then a particular class of differential equations in $\Delta_2^\circ$ as a submanifold of of $\reals^3$. Let $f \colon \delta_2^\circ \to \reals^3$, where $f(\pi) \in \reals^3$ is viewed as a random variable on $\set{0,1,2}$. Then $F(\pi) = f(\pi) - \expectat {\pi}{f(\pi)}$ is a vector field of $\Delta_2^\circ$. The differential equation is $\delta \pi = F(\pi)$, see \cite{ay|erb:2005,pistone:2010SL,pistone:2013GSI}.

\subsection{Simplex parametrizations}

Let $\Delta^0_2$ be the simplex  as a sub-variety of $\mathbb R^3$: $$\Delta^0_2=\{\bm \pi=(\pi_0,\pi_1,\pi_0) \in \mathbb R^3_{\ge}\ | \ \pi_0+\pi_1+\pi_2=1\}$$

We consider different parametrizations.

\subsubsection{Solid simplex}

\begin{equation*}
\left\{ \begin{array}{rcl}
\pi_0&=&1-\theta_1-\theta_2,\\
\pi_1&=&\theta_1,\\
\pi_2&=&\theta_2,
\end{array},
  \right.  \quad
\left\{ \begin{array}{rcl}
\theta_1&=&\pi_1,\\
\theta_2&=&\pi_2,
\end{array},
  \right.
  \quad \setof{(\theta_1,\theta_2)\in \mathbb R^2_{>}}{\theta_1+\theta_2<1}.
  \end{equation*}

The Jacobian of $\bm \theta \mapsto \bm\pi(\bm\theta))$ is
\begin{equation}
  \label{eq:J1}
  J(\bm\theta) =
  \begin{bmatrix}
    -1 & -1 \\ 1 & 0 \\ 0 & 1
  \end{bmatrix}
\end{equation}

\subsubsection{Exponential family}

The exponential family
\begin{equation}
  \label{eq:expfam}
\pi = \euler^{\theta_1 X_1 + \theta_2 X_2 - \psi(\bm\theta)}, \quad
\begin{array}{c|cc}
  i & X_1 & X_2 \\ \hline 0 & 0 & 0 \\ 1 & 1 & 0 \\ 2 & 0 & 1
\end{array}, \quad \psi(\bm\theta) = \logof{1 + \euler^{\theta_1} + \euler^{\theta_2}}
\end{equation}
gives
\begin{equation}
\left\{ \begin{array}{rcl}
\pi_0&=&(1+\euler^{\theta_1}+\euler^{\theta_2})^{-1},\\
\pi_1&=&\euler^{\theta_1}(1+\euler^{\theta_1}+\euler^{\theta_2})^{-1},\\
\pi_2&=&{\euler^{\theta_2}}(1+\euler^{\theta_1}+\euler^{\theta_2})^{-1},
\end{array}
  \right.,  \quad
\left\{ \begin{array}{rcl}
\theta_1&=&\log \frac{\pi_1}{\pi_0},\\
\theta_2&=&\log \frac{\pi_2}{\pi_0},
\end{array}
  \right.,
  \quad \{(\theta_1,\theta_2)\in \mathbb R^2\}.
  \end{equation}

As $\pi_i(\bm\theta) = \frac{\partial}{\partial \theta_i} \psi(\bm\theta)$, $i=1,2$, then $(\pi_1,\pi_2)$ (the solid simplex parameters) are the \emph{expectation parameters} of the exponential family \eqref{eq:expfam}.

The Jacobian is
\begin{equation}
  \label{eq:J2}
J(\bm\theta) =
(1+\euler^{\theta_1}+\euler^{\theta_2})^{-2}
\begin{bmatrix}
  -\euler^{\theta_1} &  - \euler^{\theta_2} \\
\euler^{\theta_{1}}(1+ \euler^{\theta_{2}}) & - 	\euler^{\theta_{1} + \theta_{2}}  \\
-\euler^{\theta_{1} + \theta_{2}} &
\euler^{\theta_{2}}(1 + \euler^{\theta_{1}} )
\end{bmatrix}
 =
\begin{bmatrix}
  -\pi_0\pi_1 & -\pi_0\pi_2 \\ \pi_1(1 - \pi_1) & -\pi_1\pi_2 \\ -\pi_1\pi_2 & \pi_2(1-\pi_2)
\end{bmatrix}.
\end{equation}
Notice that, if $Z(\bm\theta) = \euler^{\psi(\bm\theta)}$, then
\begin{equation}
  \label{eq:JvsH}
J(\bm\theta) =
\begin{bmatrix}
  \nabla Z(\bm\theta)^{-1} \\ \hessianof{\psi(\bm\theta)}
\end{bmatrix}
=
\begin{bmatrix}
  \nabla Z(\bm\theta)^{-1} \\ \varat{\bm\theta}{X_1,X_2}
\end{bmatrix}.
\end{equation}

\subsubsection{Projective parametrization}

\begin{equation*}
\left\{ \begin{array}{rcl}
\pi_0&=&(1-\theta_1-\theta_2)^{-1},\\
\pi_1&=& {\theta_1}(1-\theta_1-\theta_2)^{-1},\\
\pi_2&=& {\theta_2} (1-\theta_1-\theta_2)^{-1},
\end{array}
  \right.,  \quad
\left\{ \begin{array}{rcl}
\theta_1&=&\frac{\pi_1}{\pi_0},\\
\theta_2&=& \frac{\pi_2}{\pi_0},
\end{array}
  \right.,
  \quad \set{(\theta_1,\theta_2) \in \mathbb R^2_{>}}.
  \end{equation*}

The Jacobian is
\begin{equation}
  \label{eq:J3}
  J(\bm\theta) =\left(1-{\theta_1}-{\theta_2}\right)^{-2}
  \begin{bmatrix}
    -1 &-1\\
1-{\theta_2}& -{\theta_1}\\
-{\theta_2}& 1-{\theta_1}
  \end{bmatrix}
=
\begin{bmatrix}
  -\pi_0^2 & -\pi_0^2 \\ \pi_0(\pi_0 - \pi_2) & -\pi_0\pi_1 \\ -\pi_0\pi_2 & \pi_0(\pi_0-\pi_1)
\end{bmatrix}
\end{equation}

\subsection{Differential equations in different parametrizations}

A differential equation $\frac d {dt}\bm{\pi}(t)=\bm{F}(\bm{\pi}(t))$  in the simplex $\Delta_2^\circ$, considered as sub-variety of $\reals^3$, computed componentwise, has the form:
 \begin{equation}\label{eq:equation}
\left\{ \begin{array}{rcl}
\dot{\pi}_0(t)&=&F_0\left(\pi_0(t),\pi_1(t),\pi_2(t)\right),\\
\dot{\pi}_1(t)&=&F_1\left(\pi_0(t),\pi_1(t),\pi_2(t)\right),\\
\dot{\pi}_2(t)&=&F_2\left(\pi_0(t),\pi_1(t),\pi_2(t)\right),
\end{array}
  \right.
  \end{equation}
where $\bm F$  has to be orthogonal to the constant vector, i.e. $\bm{F \cdot 1}=0$, or $F_0+F_1+F_2=0$. Because of this condition, it is enough to consider
 \begin{equation}\label{eq:shorteqn}
\left\{ \begin{array}{rcl}
\dot{\pi}_1(t)&=&F_1\left(1-\pi_1(t)-\pi_2(t),\pi_1(t),\pi_2(t)\right),\\
\dot{\pi}_2(t)&=&F_2\left(1-\pi_1(t)-\pi_2(t),\pi_1(t),\pi_2(t)\right).
\end{array}
  \right.
  \end{equation}

The differential equation \eqref{eq:equation}, written in the parameter $\bm{\theta}$  has the form
$\frac d {dt}\bm{\pi}(\bm{\theta}(t))=\bm{F}(\bm{\pi}(\bm{\theta}(t)))$, i.e.:
\begin{equation}\label{eq-diff}
\frac{\partial}{\partial \theta_1} \bm{\pi}(\bm{\theta}(t))\ \dot{\theta}_1(t) +
\frac{\partial}{\partial \theta_2} \bm{\pi}(\bm{\theta}(t))\ \dot{\theta}_2(t)= J\bm\pi(\bm\theta(t)) \bm{\dot\theta}(t) = \bm{F}(\bm{\pi}(\bm{\theta}(t)))
\end{equation}

\subsubsection{Solid simplex} In this case, the Jacobian is \eqref{eq:J1} and

We have: $$\frac{\partial}{\partial \theta_1} \bm{\pi}(\bm{\theta}(t)) =(-1,1,0)^t \qquad \frac{\partial}{\partial \theta_2} \bm{\pi}(\bm{\theta}(t)) =(-1,0,1)^t$$ and Eq. \eqref{eq-diff} in matrix form become:
\begin{equation*}
\begin{bmatrix}
-1&-1\\
1&0\\
0&1
\end{bmatrix}
\begin{bmatrix}
\dot{\theta}_1\\
\dot{\theta}_2
\end{bmatrix} =
\begin{bmatrix}
F_0\\
F_1\\
F_2
\end{bmatrix},
\end{equation*}
which implies: $\dot{\theta}_1=F_1$ and $\dot{\theta}_2=F_2$, i.e. Eq. \eqref{eq:shorteqn}.

The choice to leave out $F_0$ is of course arbitrary.

\subsubsection{Exponential family} In this case the Jacobian is \eqref{eq:J2}, hence Eq. \eqref{eq-diff} becomes
\begin{equation*}
(1+\euler^{\theta_1}+\euler^{\theta_2})^{-2}
\begin{bmatrix}
  -\euler^{\theta_1} &  - \euler^{\theta_2} \\
\euler^{\theta_{1}}(1+ \euler^{\theta_{2}}) &  -	\euler^{\theta_{1} + \theta_{2}}  \\
-\euler^{\theta_{1} + \theta_{2}} &
\euler^{\theta_{2}}(1 + \euler^{\theta_{1}} )
\end{bmatrix}
\begin{bmatrix}
\dot{\theta}_1\\
\dot{\theta}_2
\end{bmatrix} =
\begin{bmatrix}
F_0\\
F_1\\
F_2
\end{bmatrix}.
\end{equation*}

We need to compute the inverse of the lower block, i.e.

\begin{align}
  \hessianof{\psi(\bm\theta)}^{-1} &=
   \left(1 + \euler^{\theta_1} + \euler^{\theta_2}\right)^{2}
  \begin{bmatrix}
\euler^{\theta_1}\left(1+\euler^{\theta_2}\right)&
-\euler^{\theta_1}\euler^{\theta_2}\\
-\euler^{\theta_1}\euler^{\theta_2}&
\euler^{\theta_2}\left(1+\euler^{\theta_1}\right)
  \end{bmatrix}^{-1} \notag \\
&= \frac{1 + \euler^{\theta_1} + \euler^{\theta_2}}{\euler^{\theta_1+\theta_2}}
  \begin{bmatrix}
\euler^{\theta_2}\left(1+\euler^{\theta_1}\right)&
\euler^{\theta_1+\theta_2}\\
\euler^{\theta_1+\theta_2}&
\euler^{\theta_1}\left(1+\euler^{\theta_2}\right)
  \end{bmatrix} \notag \\
&= \left(1 + \euler^{\theta_1} + \euler^{\theta_2}\right)
  \begin{bmatrix}
\frac{1+\euler^{\theta_1}}{\euler^{\theta_1}}&
1 \\
1 &
\frac{1+\euler^{\theta_2}}{\euler^{\theta_2}}
  \end{bmatrix}. \notag
\end{align}

The differential equation in the parameters $\bm\theta$ is

\begin{equation*}
\left\{\begin{array}{rcl rcll }
\dot{\theta}_1&=&\left(1 + \euler^{\theta_1} + \euler^{\theta_2}\right)\frac {1+\euler^{\theta_1}}{\euler^{\theta_1}}& F_1(\bm\theta)&+ & \left(1 + \euler^{\theta_1} + \euler^{\theta_2}\right)&F_2(\bm\theta),\\
\dot{\theta}_2&=&\left(1 + \euler^{\theta_1} + \euler^{\theta_2}\right) &F_1(\bm\theta)&+ &\left(1 + \euler^{\theta_1} + \euler^{\theta_2}\right)
\frac {1+\euler^{\theta_2}}{\euler^{\theta_2}}&F_2(\bm\theta),\end{array}\right.
\end{equation*}
or, written in $\pi$:
\begin{equation}\label{eq-dif-exp-pi}
\left\{\begin{array}{rcl rcll }
\dot{\theta}_1&=&\left(\frac {1}{\pi_1}+\frac {1}{\pi_0}\right)& F_1&+ & \frac {1}{\pi_0}&F_2\\
\dot{\theta}_2&=&\frac {1}{\pi_0} &F_1&+ &\left(\frac {1}{\pi_2}+\frac {1}{\pi_0}\right)
&F_2\end{array}\right.
\end{equation}

\subsubsection{Projective parametrization}

Eq. \eqref{eq-diff} in matrix form becomes:
\begin{equation*} \left(1-{\theta_1}-{\theta_2}\right)^{-2}
\left( \begin{array}{r r}
-1&-1\\
1-{\theta_2}&-{\theta_1}\\
-{\theta_2}&1-{\theta_1}
\end{array}\right)\
\left( \begin{array}{r}
\dot{\theta}_1\\
\dot{\theta}_2
\end{array}\right) =
\left( \begin{array}{r}
F_0\\
F_1\\
F_2
\end{array}\right)
\end{equation*}
whose solution in $\dot{\theta}_1,\dot{\theta}_2$ is:
\begin{equation*}
\left\{\begin{array}{rcl rcll }
\dot{\theta}_1&=&\left(1-\theta_1\right)\left(1-{\theta_1}-{\theta_2}\right)& F_1&+ & \theta_1\left(1-{\theta_1}-{\theta_2}\right) &F_2\\
\dot{\theta}_2&=&\theta_2\left(1-{\theta_1}-{\theta_2}\right) &F_1&+ &\left(1-{\theta_1}-{\theta_2}\right)
\left(1-{\theta_2}\right)&F_2\end{array}\right.
\end{equation*}
and written in $\pi$:
\begin{equation*}
\left\{\begin{array}{rcl rcll }
\dot{\theta}_1&=&\left(1-\pi_2\right)\pi_0^{-2}& F_1&+ & \pi_1\pi_0^{-2} &F_2\\
\dot{\theta}_2&=&\pi_2\pi_0^{-2} &F_1&+ &\left(1-\pi_1\right)\pi_0^{-2}&F_2\end{array}\right.
\end{equation*}

Notice that, in each parametrization, the condition $F_0+F_1+F_2=0$ is verified.

\section{Replicator equations}

The differential equation of the replicator, see \cite{ay|erb:2005}, is:
$$
\dot{x}_i=x_i\left(f_i(\bm{x})-\bm{x} \cdot \bm{f} \right)$$
On the simplex we have:
\begin{equation} \label{eq-repl}
\left\{ \begin{array}{rcl}
F_1(\bm{\pi})= \pi_1 \left(f_1(\bm{\pi})- \bm{\pi} \cdot \bm{f} \right)\\
F_2(\bm{\pi})= \pi_2 \left(f_2(\bm{\pi})- \bm{\pi} \cdot \bm{f} \right)
\end{array} \right.
\end{equation}
where the expected value of $f$ is
\begin{equation}
  \label{eq:3}
  \bm{\pi} \cdot \bm{f}= f_0\pi_0 + f_1\pi_1+ f_2\pi_2 = f_0+(f_1-f_0)\pi_1+(f_2-f_0)\pi_2,
\end{equation}
so that
\begin{align}
f_1 - \bm{\pi} \cdot \bm{f} &= (1-\pi_1)(f_1-f_0) - \pi_2 (f_2-f_0),\\
f_2 - \bm{\pi} \cdot \bm{f} &= -\pi_1(f_1-f_0) + (1-\pi_2)(f_2-f_0),
\end{align}
and Eq. \eqref{eq-repl} becomes:
\begin{equation}
\left\{ \begin{array}{rcl}
F_1= \pi_1 \left((f_1-f_0)(1- \pi_1)-(f_2-f_0) \pi_2 \right),\\
F_2= \pi_2 \left(-(f_1-f_0) \pi_1 + (f_2-f_0)(1- \pi_2)\right),
\end{array} \right.
\end{equation}
or
\begin{equation}
  \begin{bmatrix}
    F_1 \\ F_2
  \end{bmatrix} =
  \begin{bmatrix}
    \pi_1(1-\pi_1) & -\pi_1\pi_2 \\ -\pi_1\pi_2 & \pi_2(1-\pi_2)
  \end{bmatrix}
  \begin{bmatrix}
    f_1 - f_0 \\ f_2 - f_0
  \end{bmatrix}.
\end{equation}

Replacing  $F_1$ and $F_2$ in the differential equations with the different parametrizations, written as a function of $\bm \pi$,  we have, in the exponential parametrization,
\begin{multline}
  \begin{bmatrix}
    \frac1{\pi_0}+\frac1{\pi_1} & \frac1{\pi_0} \\
\frac1{\pi_0} & \frac1{\pi_0}+\frac1{\pi_2}
  \end{bmatrix}
\begin{bmatrix}
    \pi_1(1-\pi_1) & -\pi_1\pi_2 \\ -\pi_1\pi_2 & \pi_2(1-\pi_2)
  \end{bmatrix} \\ =
\left(\frac1{\pi_0}
\begin{bmatrix}
  1 & 1 \\ 1 & 1
\end{bmatrix} +
\begin{bmatrix}
  \pi_1^{-1} & 0 \\ 0 & \pi_2^{-1}
\end{bmatrix}\right)
\begin{bmatrix}
    \pi_1(1-\pi_1) & -\pi_1\pi_2 \\ -\pi_1\pi_2 & \pi_2(1-\pi_2)
  \end{bmatrix} \\ =
  \begin{bmatrix}
    \pi_1 & \pi_2 \\ \pi_1 & \pi_2
  \end{bmatrix} +
  \begin{bmatrix}
    1-\pi_1 & -\pi_2 \\ -\pi_1 & 1 - \pi_2
  \end{bmatrix}
= I_2
\end{multline}
so that the differential equations are
\begin{equation}\label{eq:diffexp}
\left\{\begin{aligned}
\dot{\theta}_1 &= f_1-f_0 \\
\dot{\theta}_2 &= f_2-f_0
\end{aligned}
\right.
\end{equation}

In the projective parametrization
\begin{multline}
  \begin{bmatrix}
    (1-\pi_2)\pi_0^{-2} & \pi_1\pi_0^{-2} \\
\pi_2\pi_0^{-2} & (1-\pi_1)\pi_0^{-2}
  \end{bmatrix}
\begin{bmatrix}
    \pi_1(1-\pi_1) & -\pi_1\pi_2 \\ -\pi_1\pi_2 & \pi_2(1-\pi_2)
  \end{bmatrix} \\ =
\frac1{\pi_0^2}  \begin{bmatrix}
    1-\pi_2 & \pi_1 \\
\pi_2 & 1-\pi_1
  \end{bmatrix}
\begin{bmatrix}
    1-\pi_1 & -\pi_1 \\ - \pi_2 & 1-\pi_2
  \end{bmatrix}
  \begin{bmatrix}
    \pi_1 & 0 \\ 0 & \pi_2
  \end{bmatrix}
 =
\begin{bmatrix}
  \frac{\pi_1}{\pi_0} & 0 \\ 0 & \frac{\pi_2}{\pi_0}
\end{bmatrix}
\end{multline}
hence the differential equations are

\begin{equation}
\left\{\begin{aligned}
\dot{\theta}_1 &= \theta_1 (f_1-f_0) \\
\dot{\theta}_2 &= \theta_2 (f_2-f_0)
\end{aligned}
\right.
\end{equation}

\section{Second order calculus}
\label{sec:hessian}

The expectation parameters are
\begin{equation*}
  \eta_i(\bm\theta) = \frac{\euler^{\theta_i}}{1+\sum_{j=1}^d \euler^{\theta_j}}
\end{equation*}
and the information matrix is
\begin{equation*}
  I(\bm\theta) = \diagof{\bm\eta(\bm\theta)} - \bm\eta(\bm\theta)\bm\eta(\bm\theta)'
\end{equation*}

The precision matrix is
\begin{equation*}
  I^{-1} = \diagof{\bm\eta}^{-1} + (1-\absoluteval{\bm\eta})^{-1} \bm 1 \bm 1'.
\end{equation*}
In fact
\begin{multline*}
\left(\diagof{\bm\eta} - \bm\eta\bm\eta'\right)\left(\diagof{\bm\eta}^{-1} + (1-\absoluteval{\bm\eta})^{-1} \bm 1 \bm 1'\right) = \\
I + (1-\absoluteval{\bm\eta})^{-1} \diagof{\bm\eta}\bm 1 \bm 1' - \bm\eta\bm\eta' \diagof{\bm\eta}^{-1} - (1-\absoluteval{\bm\eta})^{-1} \bm\eta\bm\eta' \bm 1 \bm 1' = \\
I + (1-\absoluteval{\bm\eta})^{-1} \bm\eta \bm 1' - \bm\eta\bm 1' - \absoluteval{\bm\eta} (1-\absoluteval{\bm\eta})^{-1} \bm\eta \bm 1' = I.\\
\end{multline*}

The derivative in the direction $\bm h$ of $\bm\eta \mapsto \diagof{\bm\eta} - \bm\eta\bm\eta'$ at $\bm\eta$ is
\begin{equation*}
  \bm h \mapsto \diagof{\bm h} - \bm h \bm \eta' - \bm \eta \bm h',
\end{equation*}
hence, composing with
\begin{equation*}
  \partial_i \bm\eta(\bm\theta) = I(\bm\theta) \bm e_i = \left(\diagof{\bm\eta(\bm\theta)} - \bm\eta(\bm\theta)\bm\eta(\bm\theta)'\right) \bm e_i = \eta_i(\bm\theta) \left(\bm e_i - \bm\eta(\bm\theta) \right),
\end{equation*}
we obtain
\begin{align*}
  \partial_i I(\bm\theta) &= \diagof{\eta_i(\bm\theta) \left(\bm e_i - \bm\eta(\bm\theta) \right)} - \eta_i(\bm\theta) \left(\bm e_i - \bm\eta(\bm\theta) \right) \bm \eta(\bm\theta)' - \bm \eta \eta_i(\bm\theta) \left(\bm e_i - \bm\eta(\bm\theta) \right)' \\
&= \eta_i(\bm\theta) \left(\diagof{\left(\bm e_i - \bm\eta(\bm\theta) \right)} - \left(\bm e_i - \bm\eta(\bm\theta) \right) \bm \eta(\bm\theta)' - \bm \eta(\bm\theta) \left(\bm e_i - \bm\eta(\bm\theta) \right)'\right).
\end{align*}

The following equality is to be used below.

\begin{multline*}
  I^{-1} \partial_i I = \\
\left(\diagof{\bm\eta}^{-1} + (1-\absoluteval{\bm\eta})^{-1} \bm 1 \bm 1'\right)\left(\eta_i \left(\diagof{\left(\bm e_i - \bm\eta \right)} - \left(\bm e_i - \bm\eta \right) \bm \eta' - \bm \eta \left(\bm e_i - \bm\eta \right)'\right)\right)
= \\ \eta_i \diagof{\bm\eta}^{-1} \left(\diagof{\left(\bm e_i - \bm\eta \right)} - \left(\bm e_i - \bm\eta \right) \bm \eta' - \bm \eta \left(\bm e_i - \bm\eta \right)'\right) + \\ \eta_i (1 - \absoluteval{\bm\eta})^{-1} \bm 1 \bm 1' \left(\diagof{\left(\bm e_i - \bm\eta \right)} - \left(\bm e_i - \bm\eta \right) \bm \eta' - \bm \eta \left(\bm e_i - \bm\eta \right)'\right) = \\
\end{multline*}

We conclude by briefly reviewing the computation of the metric connection (Levi-Civita connection) which is required by e.g., the computation of the Riemannian Hessian.

Let $\bm{\pi}(t)$ be an univariate statistical model. Let $F(\bm{\pi}(t))$ and $G(\bm{\pi}(t))$ be centered pivotal quantities (vector fields of the statistical model). The variation of $\textrm{Cov}(F(\bm{\pi}(t)),G(\bm{\pi}(t)))$ in the time is:
\begin{multline*}
\frac d {dt}\left(\textrm{Cov}(F(\bm{\pi}(t)),G(\bm{\pi}(t)))\right)= \frac d {dt} \langle F(\bm{\pi}(t)), G(\bm{\pi}(t)) \rangle_{\bm{\pi}(t)}= \\ \frac d {dt} \sum_{i,j} F_i(\bm{\pi}(t)) G_i(\bm{\pi}(t)) I_{ij}(\bm{\pi}(t))=\\
\sum_{i,j}\left( \frac d {dt} F_i\right)  G_i I_{ij}+\sum_{i,j} F_i \left( \frac d {dt}  G_i \right)I_{ij}+ \sum_{i,j} F_i G_i  \left( \frac d {dt} I_{ij}\right)=\\
 \langle  \frac d {dt} F, G  \rangle_{\bm{\pi}(t)}+\langle   F, \frac d {dt}G \rangle_{\bm{\pi}(t)}+
 \frac 1 2  \langle I^{-1} \frac d {dt}\left(I \right) F, G  \rangle_{\bm{\pi}(t)}+
  \frac 1 2  \langle  F,I^{-1} \frac d {dt}\left(I \right) G \rangle_{\bm{\pi}(t)}=\\
   \langle  \frac d {dt} F+\frac 1 2 I^{-1} \frac d {dt}\left(I \right)F,G \rangle_{\bm{\pi}(t)}+
   \langle  F, \frac d {dt}G +\frac 1 2 I^{-1}  \frac d {dt}\left(I \right)G\rangle_{\bm{\pi}(t)},
\end{multline*}
see \cite{docarmo:1992,pistone:2013GSI}.

The last line can be properly written as:
\begin{multline*}
   \langle  \frac d {dt}\left(F_i(\bm{\pi}(t))\right)_{i=1}^n +\frac 1 2 I^{-1}(\bm{\pi}(t)) \frac d {dt}\left(I(\bm{\pi}(t)) \right)F(\bm{\pi}(t)),G \rangle_{\bm{\pi}(t)}+\\
   \langle  F(\bm{\pi}(t)), \frac d {dt}\left(G_i(\bm{\pi}(t))\right)_{i=1}^n +\frac 1 2 I^{-1}(\bm{\pi}(t))  \frac d {dt}\left(I(\bm{\pi}(t)) \right)G\rangle_{\bm{\pi}(t)},
\end{multline*}
which defines an operator (affine connection \cite{docarmo:1992} or metric derivative \cite{lang:1995})

\begin{equation*}
  (F,\bm\alpha) \mapsto \nabla F \bm\alpha + \frac12 I^{-1} (\nabla I \bm\alpha) F
\end{equation*}

If the model is such that $ I^{-1}(\bm{\pi}(t)) \frac d {dt}\left(I(\bm{\pi}(t))\right)=0 $, then
\begin{multline*}
\frac d {dt}\left(\textrm{Cov}(F(\bm{\pi}(t)),G(\bm{\pi}(t)))\right)=  \langle  \frac d {dt}\left(F_i(\bm{\pi}(t))\right)_{i=1}^n ,G \rangle_{\bm{\pi}(t)}+
   \langle  F(\bm{\pi}(t)), \frac d {dt}\left(G_i(\bm{\pi}(t))\right)_{i=1}^n \rangle_{\bm{\pi}(t)}.
   \end{multline*}

\end{document}